\documentclass[11pt,a4paper,reqno]{amsart} \usepackage{amsmath,amssymb,tikz,color}

\usepackage[vcentermath]{youngtab}
\usepackage{youngtab}
\usetikzlibrary{chains}
\usepackage{subfigure}

\newcommand{\sd}[4]{$\underset{\hskip -5pt\scriptscriptstyle #3,#2}{#1^{#4}}$}
\newcommand{\sdu}[4]{$\underset{\hskip -5pt\scriptscriptstyle #3,#2}{\underline{#1}^{#4}}$}
\newcommand{\hs}{\hskip 5pt}
\newcommand{\Rec}{\mathsf{Rec}}
\newcommand{\MinRec}{\mathsf{Rec}^{\mathsf{min}}}

\newcommand{\EW}{{\mathsf{EWtab}}}
\newcommand{\DEW}{\mathsf{DecEWtab}}
\newcommand{\rowlabs}{\mathsf{rows}}
\newcommand{\collabs}{\mathsf{cols}}
\newcommand{\ttc}{\phi_{TC}}
\newcommand{\ctt}{\phi_{CT}}

\newcommand{\tts}{\phi_{TS}}
\newcommand{\CanonTop}{\mathsf{CanonTop}}
\newcommand{\minrec}{\mathsf{minrec}}
\newcommand{\dew}{\psi}
\newcommand{\RunDec}{\mathsf{RunDec}}
\newcommand{\RDEW}{\mathsf{Rec}_{\mathsf{EW}}}

\newcommand{\cls}[1]{\mathcal{S}_{#1}}
\newcommand{\UU}{\mathcal{U}}
\newcommand{\VV}{\mathcal{V}}
\newcommand{\dec}{\mathsf{dec}}
\newcommand{\inc}{\mathsf{inc}}
\newcommand{\topbef}{\prec_{T}}
\newcommand{\U}[2]{U_{#1}^{(#2)}}
\newcommand{\V}[2]{V_{#1}^{(#2)}}
\newcommand{\D}[2]{D_{#1}^{(#2)}}
\newcommand{\A}[2]{A_{#1}^{(#2)}}

\newcommand{\mycomp}[1]{\overline{#1}}
\newcommand{\rightangle}{cornersupport}

 \newtheorem{theorem}{Theorem}[section]
 \newtheorem{corollary}[theorem]{Corollary}
 \newtheorem{prop}[theorem]{Proposition}
 
 \newtheorem{lemma}[theorem]{Lemma}

 \theoremstyle{definition}
 \newtheorem{example}[theorem]{Example}

 \newtheorem{definition}[theorem]{Definition}
 
 \newtheorem{remark}[theorem]{Remark}

 \theoremstyle{remark}
 \newtheorem{algorithm}[theorem]{Algorithm}

\definecolor{mygreen}{RGB}{80,160,80}

\newcommand{\EWdiagram}[3]{
  \begin{tikzpicture}
\def\step{0.43}
\draw (0,0) node [anchor = north west]  {\young(#1)};
\foreach[count=\y] \lab in {#2}
	\draw (0.25,-\y*\step+0.05) node [anchor=east]{\tiny{\lab}};
\foreach[count=\x] \lab in {#3}
	\draw (\x*\step-0.05,0.2) node [anchor=north]{\tiny{\lab}};
      \end{tikzpicture}
}
\newcommand{\EWdiagramlab}[4]{
  \begin{tikzpicture}
\def\step{0.43}
\draw (0,0) node [anchor = north west]  {\young(#1)};
\foreach[count=\y] \lab in {#2}
	\draw (0.25,-\y*\step+0.05) node [anchor=east]{\tiny{\lab}};
\foreach[count=\x] \lab in {#3}
	\draw (\x*\step-0.05,0.2) node [anchor=north]{\tiny{\lab}};
\draw (-.25,-1) node [anchor = east]{#4};
\end{tikzpicture}
}
\newcommand{\decoEWdiagram}[6]{
  \begin{tikzpicture}
\def\step{0.43}
\draw (0,0) node [anchor = north west]  {\young(#1)};
\foreach[count=\y] \lab in {#2}
	\draw (0.25,-\y*\step+0.05) node [anchor=east]{\tiny{\lab}};
\foreach[count=\x] \lab in {#3}
	\draw (\x*\step-0.05,0.2) node [anchor=north]{\tiny{\lab}};
\foreach[count=\yy from 2] \a/\b in {#5}
	\draw (\a*\step+.22,-\yy*\step+0.275) node [anchor=north]{\tiny{\b}};
\foreach[count=\xx] \a/\b in {#6}
	\draw (\xx*\step-0.05,-\a*\step-.03) node [anchor=north]{\tiny{\b}};
\draw (-.25,-1) node [anchor = east]{#4};
\end{tikzpicture}
}

\title[The Abelian sandpile model on Ferrers graphs]{The Abelian sandpile model on Ferrers graphs --- A classification of recurrent configurations}
\author[M. Dukes]{Mark Dukes}
\address{UCD School of Mathematics and Statistics, University College Dublin, Dublin 4, Ireland} \email{mark.dukes@ccc.oxon.org}
\author[T. Selig]{Thomas Selig}
\address{Mathematics Division, Science Institute, University of Iceland, Dunhaga 5, 107 Reykjav\'ik, Iceland.} \email{selig@hi.is}
\author[J.P. Smith]{Jason P. Smith}
\address{Department of Mathematics, University of Aberdeen, Aberdeen, AB24 3FX, UK}
\email{jason.smith@abdn.ac.uk }
\author[E. Steingr\'imsson]{Einar Steingr\'imsson}
\address{Department of Computer and Information Sciences, University of Strathclyde, Glasgow G1 1XH, U.K.} 
\email{einar@alum.mit.edu}
\thanks{This work was supported by grants EP/M015874/1 and EP/M027147/1 from The Engineering and Physical Sciences Research Council.}

\date{\today}

\begin{document}

\begin{abstract}
We classify all recurrent configurations of the Abelian sandpile model (ASM) on Ferrers graphs. 
The classification is in terms of decorations of EW-tableaux, which undecorated are in bijection with the minimal recurrent configurations. 
We introduce decorated permutations, extending to decorated EW-tableaux a bijection between such tableaux and permutations, giving a direct bijection between the decorated permutations and all recurrent configurations of the ASM.    We also describe a bijection between the decorated permutations and the intransitive trees of Postnikov, the 
breadth-first search of which corresponds to a canonical toppling of the corresponding configurations. 
\end{abstract}

\maketitle

\section{Introduction}\label{sec:intro}

In this paper we classify all recurrent configurations of the Abelian sandpile model (ASM)
on the Ferrers graphs defined in \cite{evw}.  A Ferrers graph is a bipartite graph whose vertices correspond to rows and columns in a Ferrers diagram $F$, with an edge between vertices $r$ and $c$ if $F$ has a cell in row $r$ and column $c$ (see Example~\ref{ex:FerLabel,FerGr}, and Section~\ref{sec:defs} for other definitions).  

This work builds on the previous paper by three of the current authors~\cite{sss}, where a bijective correspondence was established between the minimal recurrent configurations of the ASM on Ferrers graphs 
and the so-called EW-tableaux, which are certain 0/1-fillings  of the corresponding  Ferrers diagrams. In \cite{sss} a bijection was also provided between the EW-tableaux and permutations whose excedances are determined by the shape of the corresponding Ferrers diagram.  

In order to give a general classification of the recurrent configurations, we introduce 
the notion of {\emph{decorated}} EW-tableaux. These decorated tableaux are EW-tableaux in which every row and column has a non-negative integer associated with it.  Those sequences of decorations that correspond to recurrent configurations are classified in terms of properties of the EW-tableaux that they decorate.   Moreover, we extend the bijection in \cite{sss} from EW-tableaux to permutations to produce decorated permutations, described in terms of simple statistics on the permutations, resulting in a direct bijective correspondence between these decorated permutations and all recurrent configurations on the corresponding Ferrers graphs.  

We also describe a bijection between the decorated permutations and the \emph{intransitive trees} defined by Postnikov~\cite{postnikov-intransitive-trees}. These are labeled trees where each vertex has a greater label than all its neighbors or else a smaller label than all of them.  Our bijection has the property that the canonical toppling of vertices of the corresponding graph corresponds to the breadth-first search of the tree.
 
The paper is organized as follows. In Section~\ref{sec:defs} we recall some necessary definitions related to the objects of this paper, and give some background.  In Section~\ref{sec:EWT,perm,minrec} we extend to all Ferrers graphs the notion of canonical topplings, first introduced in the case of complete bipartite graphs by Dukes and Le Borgne~\cite{dlb}. We also recall the bijection from EW-tableaux to minimal recurrent configurations of the ASM on Ferrers graphs and give an explicit description of the inverse.  This inverse relies on the notion of canonical toppling, which is closely linked to the permutation associated to a EW-tableau in \cite{sss}.  In Section~\ref{sec:recstates_decEWT} we introduce the notion of \emph{supplementary tableaux} and bring together observations and results from previous sections to culminate in the complete classification of all recurrent configurations on Ferrers graphs, in Theorem~\ref{classif}.  In Section~\ref{sec:decperms} we consider the decorated permutations that correspond (via extension of results presented in~\cite{sss}) to decorated EW-tableaux, show how the stabilization process of the ASM translates to these decorated permutations, and exhibit a bijection between the decorated permutations and the intransitive trees defined by Postnikov~\cite{postnikov-intransitive-trees}.  The breadth-first search of such a tree corresponds to the canonical toppling of the corresponding configuration on the corresponding Ferrers graph.


\section{Definitions and Background}\label{sec:defs}

\subsection{Ferrers graphs}\label{sec:FerGr}

Let $F$ be a Ferrers diagram with rows and columns labeled $0,\ldots,n$ from top right to bottom left along the south-east border; see Example~\ref{ex:FerLabel,FerGr}.
The sets of row and column labels of $F$ are denoted by $\rowlabs(F)$ and $\collabs(F)$, respectively.
Let $G=G(F)$ be the graph corresponding to $F$ thus: If $i$ is a row label and $j$ a column label of $F$, then $(i,j)$ is an edge of $G$ if and only if $F$ has a cell in row $i$ and column $j$.
Note that it is a consequence of the definition that $(i,j) \in \rowlabs(F) \times \collabs(F)$ is an edge of $G(F)$ if and only if $i<j$. 
We identify the label of a row or column of $F$ with its corresponding vertex in $G$, that is, we consider $G$ to be a graph with vertex set $\{0,\ldots,n\}$.

\begin{example}\label{ex:FerLabel,FerGr}
Let $F$ be the Ferrers diagram in Figure~\ref{fig:FerLabel,FerGr} below.
Label the lower path from top right to bottom left with the labels $0,1,\ldots ,8$. 
Then copy each column label to the top of its column and each row label to the left of its row, as shown on the right.

The row labels are $\rowlabs(F)=\{0,3,4,6\}$ and $\collabs(F)=\{1,2,5,7,8\}$.
We will always consider the row and column labels to be along the north-west boundary, as in
the diagram on the right, as we introduce a separate labeling of the 
south-east boundary in Section~\ref{sec:recstates_decEWT}.
Then $G(F)$ is the graph with vertex set $\{0,\ldots,8\}$ and edge set $E(G)$, where $(i,j) \in E(G)$ if and only if $F$ has a cell in row labeled $i$ and column labeled $j$:

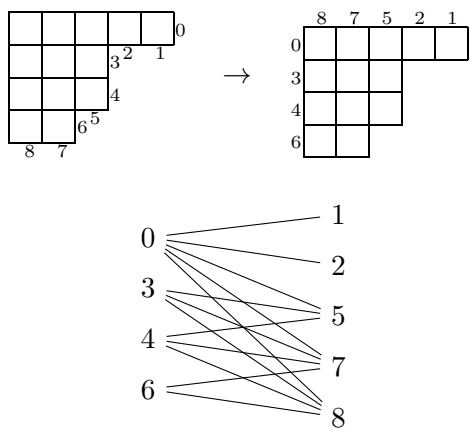
\begin{figure}[ht]
 \centering
 \subfigure{
  \begin{tikzpicture}
\def\xx{0.435}
\def\bx{0}\def\by{0}
\def\initx{0.25}\def\inity{0.37}
\def\jnitx{0.31}\def\jnity{0.47}
\def\stepx{0.43}\def\stepy{0.43}
\draw (\bx,\by) node [anchor = north east]  {\young(~~~~~,~~~,~~~,~~)};
\foreach \x/\y/\lab in {0/0/0,2/1/3,2/2/4,3/3/6}  
	\draw (\bx-\initx-\x*\stepx,\by-\inity-\y*\stepy) node [anchor=west]{\tiny{\lab}};
\foreach \x/\y/\lab in {0/0/1,1/0/2,2/2/5,3/3/7,4/3/8}
	\draw (\bx-\jnitx-\x*\stepx,\by-\jnity-\y*\stepy) node [anchor=north]{\tiny{\lab}}; 
\draw (0.7,-1) node {$\to$};
  \end{tikzpicture}
 \EWdiagram{~~~~~,~~~,~~~,~~}{0,3,4,6}{8,7,5,2,1}
 } \\
 \subfigure{
  \begin{tikzpicture}[fsnode/.style={}, ssnode/.style={}, every fit/.style={inner sep=5pt,text width=0cm}, -,shorten >= 3pt,shorten <= 3pt ]
    \begin{scope}[start chain=going below,node distance=4mm]
\foreach \i in {0,3,4,6}
  \node[fsnode,on chain,label=center:${\i}$] (f\i){};
\end{scope}
\begin{scope}[xshift=2.5cm,yshift=3mm,start chain=going below,node distance=4mm]
\foreach \i/\xcoord/\ycoord in {1,2,5,7,8}
  \node[ssnode,on chain,label=center:${\i}$] (s\i){};
\draw (f0) -- (s1); \draw (f0) -- (s2); \draw (f0) -- (s5); \draw (f0) -- (s7);
\draw (f0) -- (s8); \draw (f3) -- (s5); \draw (f3) -- (s7); \draw (f3) -- (s8);
\draw (f4) -- (s5); \draw (f4) -- (s7); \draw (f4) -- (s8); \draw (f6) -- (s7); \draw (f6) -- (s8);
	\end{scope}
  \end{tikzpicture}
 }
\caption{An example of a Ferrers diagram with its row and column labels, and the corresponding Ferrers graph.\label{fig:FerLabel,FerGr}}
\end{figure}

\end{example}

\subsection{EW-tableaux}\label{sec:EWT}

EW-tableaux were first described in \cite{evw}, in a slightly more general form, but the following version was introduced in \cite{sss}.

\begin{definition}\label{def:EWT}
A EW-tableau $T$ is a 0/1-filling of a Ferrers diagram with the following properties:
\begin{enumerate}
\item[1.] The top row of $T$ has a 1 in every cell.
\item[2.] Every other row has at least one cell containing a 0.
\item[3.] No four cells of $T$ that form the corners of a rectangle have 0s in two diagonally opposite corners and 1s in the other two.
\end{enumerate}
The \emph{size} of a EW-tableau is its semiperimeter less one, that is, one less than the sum of its numbers of rows and columns. Let $\EW(F)$ be the set of EW-tableaux on the diagram $F$.
\end{definition}
Given a EW-tableau $T$, we let $T_{ij}$ denote the entry in row $i$ and column~$j$. If  $T$ does not have a cell in row $i$ and column $j$ we write $T_{ij}=\varepsilon$.

EW-tableaux were introduced by Ehrenborg and van Willigenburg in \cite{evw}, with the difference that the all-1s row was arbitrary, in contrast to the above definition fixing it to be the top row. In particular, the authors showed that EW-tableaux are equinumerous with acyclic orientations, with a unique sink, of the corresponding Ferrers graphs, and with permutations according to their sets of excedances, that is, places $i$ such that $a_i>i$ in a permutation $a_1a_2\ldots a_n$ of the letters $1,2,\ldots,n$. 
In \cite{sss}, these tableaux were studied further, a bijection to permutations presented, and various enumerative results established. It was also shown that the set of EW-tableaux on a Ferrers diagram $F$ is in one-to-one correspondence with the set of minimal recurrent configurations of the Abelian sandpile model on the Ferrers graph $G(F)$ (see Section~\ref{sec:map_ctt} for more details on this model). 

\begin{example}\label{firsttabex}
Consider the EW-tableau $T$ below. The column labels are $\collabs(T) = \{1,2,5,7,8\}$ and 
$\rowlabs(T) = \{0,3,4,6\}$.
\begin{center}
\EWdiagram{11111,010,011,01}{0,3,4,6}{8,7,5,2,1}
\end{center}
Some of its entries are $T_{37}=1$, $T_{68}=0$, and $T_{62}=\varepsilon$.
\end{example}

\begin{example}\label{exampleone}
Let $F$ be the Ferrers diagram of shape (3,2,1). 
The row labels, read from top to bottom, are $0,2,4$, and the column labels, read from left to right, are $5,3,1$. 
There are three EW-tableaux of this shape:
$$\EW(F)=\left\{  
  {\scriptsize{\young(111,00,0)}}, {\scriptsize{\young(111,01,0)}}, {\scriptsize{\young(111,10,0)}} \right\}.$$
\end{example}

\subsection{The sandpile model on Ferrers graphs}\label{sec:ASM}
The Abelian sandpile model (ASM) is defined on any undirected graph $G$ with a designated vertex~$s$ called the \emph{sink}.  A \emph{configuration} on $G$ is an assignment of non-negative integers to the non-sink vertices of $G$:
$$c:V(G)\backslash\{s\} \mapsto \mathbb{N}.$$
The number $c(v)$ is sometimes referred to as the number of \emph{grains} at vertex~$v$, or as the \emph{height} of $v$.  Our Ferrers graphs have vertex set $\{0,1,\ldots,n\}$, the labels of the rows and columns of the corresponding Ferrers diagram, and the sink vertex is always $0$, corresponding to the top row. We will write a configuration as $c=(c_1,\ldots,c_n)$.  Given a configuration $c$, a vertex $v$ is said to be {\emph{stable}} if the number of grains at $v$ is strictly smaller than the degree of $v$. Otherwise $v$ is {\emph{unstable}}.  A configuration is {\emph{stable}} if all non-sink vertices are stable.

If a vertex is unstable then it may {\emph{topple}}, which means the vertex donates one grain to each of its neighbors.  Since the sink has no height associated with it, it absorbs grains and that is how grains may exit the system.  Given this, it is possible to show (see for instance \cite[Section 5.2]{Dhar}) that starting from any configuration $c$ and toppling unstable vertices, one eventually reaches a stable configuration $c'$.  Moreover, $c'$ does not depend on the order in which vertices are toppled in this sequence. We call $c'$ the \emph{stabilization} of $c$.

Starting from the empty configuration, one may indefinitely add any number of grains to any vertices and topple vertices should they become unstable.  Certain stable configurations will appear again and again, that is, they \emph{recur}, while other stable configurations will never appear again.  
More precisely, a configuration is recurrent if, and only if, it can be achieved by adding grains to the maximally stable configuration (the configuration in which every vertex is one shy of toppling) and stabilizing.
These {\emph{recurrent configurations}} are the ones that appear in the long term limit of the system and are consequently of most interest.  Let $\Rec(G)$ be the set of recurrent configurations on $G$, and let $\MinRec(G)$ be the \emph{minimally recurrent} configurations, that is, those whose sums of heights are minimal.  

In \cite[Section 6]{Dhar}, Dhar describes the so-called \emph{burning algorithm}, which establishes in linear time whether a given stable configuration is recurrent. We recall the result here.

\begin{prop}[\cite{Dhar}, Section 6.1]\label{pro:DharBurning}
Let $G$ be a graph with sink $s$, and $c$ a stable configuration on $G$. Then $c$ is recurrent if and only if there exists an ordering $v_0=s,v_1,\ldots,v_n$ of the vertices of $G$ such that, starting from $c$, for any $i \geq 1$, toppling the vertices $v_0,\ldots,v_{i-1}$ causes the vertex $v_i$ to become unstable. Moreover, if such a sequence exists, then toppling $v_0,\ldots,v_n$ returns the initial configuration $c$.
\end{prop}

The question of determining all recurrent configurations for the ASM on a given graph is however less straightforward. 
The computations required for going through all possible stable configurations on a graph and checking whether the above criterion holds are significant, since
the number of  possible stable configurations on a graph equals the product of the degrees of non-sink vertices of that graph, so for the complete graph $K_n$, for example, 
this is~$(n-1)^{n-1}$.

Using the matrix-tree theorem (see e.g. \cite{Biggs93}), one can show 
that the number of recurrent configurations on a graph $G$ equals the number of spanning trees of $G$. There are various bijective proofs of this result: In \cite{Red} the author provides a non-canonical bijection between recurrent configurations and spanning trees, while the authors of \cite{Ber} and \cite{clb} both provided refined versions which enumerate recurrent configurations according to the \emph{level} statistic, where the level of a configuration is the sum of its entries, plus the degree of the sink minus the number of edges of the graph.
In \cite{ds} the authors study the problem of classifying recurrent configurations of the ASM, and present several results in the case where the graph may be decomposed in a variety of ways.

This problem has also been approached for various specific families of graphs, with some interesting links to other combinatorial structures.
\begin{itemize}
\item Trees and cycles are straightforward (see e.g. \cite{Red}).
\item Complete graphs $K_n$ \cite{cr}. The authors show that the set of recurrent configurations is in one-to-one correspondence with the set of \emph{parking functions}.
\item Complete multipartite graphs with a dominating sink $K_{1,p_1,\ldots,p_k}$ \cite{cp}. The authors show that the set of recurrent configurations is in one-to-one correspondence with the set of \emph{$(p_1,\ldots,p_k)$-parking functions}.
\item Complete bipartite graphs $K_{m,n}$ where the sink is in a particular one of the two parts \cite{dlb} (see also \cite{aadhlb,aadlb}). The authors show that the set of ordered recurrent configurations is in one-to-one correspondence with the set of \emph{parallelogram polyominoes} on an $m \times n$ grid.
\end{itemize}
In this paper, we classify all recurrent configurations for the ASM on Ferrers graphs. These graphs are a class of bipartite graphs which include the complete bipartite graphs (corresponding to rectangular Ferrers diagrams). As such, this work can be seen as complementing and extending that of \cite{aadhlb,aadlb,dlb,sss}.


\section{EW-tableaux, permutations, and minimal recurrent configurations of the ASM on Ferrers graphs}\label{sec:EWT,perm,minrec}

\subsection{Mapping EW-tableaux to minimal recurrent configurations}\label{sec:map_ttc}

In \cite[Theorem~22]{sss}, the authors established a bijection 
$$
\ttc: \EW(F) \to \MinRec(G(F)),
$$
defined as follows:
\begin{definition}\label{ttcdefn}
Given $T \in \EW(F)$, let $\ttc(T)$ be the sandpile configuration $(c_1,\ldots,c_n)$ on the graph $G(F)$
where 
$$c_i=\begin{cases} \mbox{number of 1s in row labeled $i$ of $T$} &  \mbox{if $i \in \rowlabs(F)$},\\ \mbox{number of 0s in column labeled $i$ of $T$} & \mbox{if $i \in \collabs(F)$.} \end{cases}$$
\end{definition}

\begin{example}\label{examplethree}
Recall the following EW-tableau $T$ from Example~\ref{firsttabex}:
\begin{center}
\EWdiagram{11111,010,011,01}{0,3,4,6}{8,7,5,2,1}
\end{center}
We form $\ttc(T) = (c_1,\ldots,c_8)$ as follows:
The row labels are $\rowlabs(T) = \{0,3,4,6\}$ and the column labels are $\collabs(T)=\{1,2,5,7,8\}$.
The quantity $c_3$ is the number of 1s in row labeled 3 (which is the second row) of $T$, so $c_3 = 1$.
Similarly, $c_4=2$ and $c_6=1$.
The quantity $c_1$ is the number of 0s in column labeled 1 of $T$ (the rightmost column), so $c_1=0$.
Similarly, $c_2=0$, $c_5=1$, $c_7=0$ and $c_8=3$.
Therefore
$\ttc(T) = (0,0,1,2,1,1,0,3).$
\end{example}

\subsection{Canonical toppling}\label{sec:canontop}

In this section we recall the notion of \emph{canonical topplings} for Ferrers graphs,
which was introduced in the case of complete bipartite graphs by Dukes and Le Borgne~\cite{dlb}. 
This notion helps us give explicit descriptions of the inverse of $\ttc$ (in Section~\ref{sec:map_ctt}) and of the
bijection between EW-tableaux and permutations introduced in \cite{sss} (recalled also at the end of this section).

\begin{definition}\label{def:canontop}
Given $c \in \Rec(G)$ where $G=G(F)$, let
$$\CanonTop (c) := (\U{c}{0},\V{c}{1},\U{c}{1},\V{c}{2}\ldots)$$
be the sequence of canonical topplings that occur as a result of toppling the sink vertex $0$ followed by alternately toppling the sets of all the unstable vertices in $\collabs(F)$ and $\rowlabs(F)$.  Note that $\U{c}{0}=\{0\}$, $\U{c}{i} \subseteq \rowlabs(F)$ and $\V{c}{i} \subseteq \collabs(F)$ for all $i$, and that by Proposition~\ref{pro:DharBurning} $\CanonTop(c)$ forms an ordered partition of the vertex set $\{0,\ldots,n\}$.
\end{definition}

\begin{example}\label{ex:canontop}
  \newcommand{\ul}[1]{\underline{#1}} 
Let $T$ be the tableau from Example~\ref{examplethree} so that $\ttc(T) = (0,0,1,2,1,1,0,3)$.  For ease of reference, let us note that the array recording the least unstable height of each vertex is $(1,1,3,3,3,2,4,4)$.  Also, let us underline those parts that correspond to the same part of the graph as the sink. So
$$
\ttc(T)=c^{(0)} = (0,0,\ul{1},\ul{2},1,\ul{1},0,3).
$$
Topple the sink, which is vertex $0$, so $\U{c}{0} = \{0\}$.  This causes vertices $1,2,5,7,8$ to inherit a single grain each.  Compare with the array $(1,1,3,3,3,2,4,4)$ and mark with an asterisk, and subsequently topple, those vertices whose values equal or exceed these, and repeat until there are no vertices to topple:
\begin{center}\begin{tikzpicture}[scale=1.1]
\node (c1) at (0,4){$c^{(1)} =  (1^*,1^*,\ul{1},\ul{2},2,\ul{1},1,4^*)$};
\node (c2) at (0,3){$c^{(2)} =  (0,0,\ul{2},\ul{3}^*,2,\ul{2}^*,1,0)$};
\node (c3) at (0,2){$c^{(3)} =  (0,0,\ul{2},\ul{0},3^*,\ul{0},3,2)$};
\node (c4) at (0,1){$c^{(4)} =  (0,0,\ul{3}^*,\ul{1},0,\ul{0},3,2)$};
\node (c5) at (0,0){$c^{(5)} =  (0,0,\ul{0},\ul{1},1,\ul{0},4^*,3)$};
\node (c6) at (0,-1){$c^{(5)} =  (0,0,\ul{1},\ul{2},1,\ul{1},0,3)=\ttc(T)$};
\draw[->] (c1) -- (c2);
\draw[->] (c2) -- (c3);
\draw[->] (c3) -- (c4);
\draw[->] (c4) -- (c5);
\draw[->] (c5) -- (c6);
\node (a1) at (3,3.5){\tiny $\V{c}{1} = \{1,2,8\}$};
\node (a2) at (3,2.5){\tiny $\U{c}{1} = \{4,6\}$};
\node (a3) at (3,1.5){\tiny $\V{c}{2} = \{5\}$};
\node (a4) at (3,.5){\tiny $\U{c}{2} = \{3\}$};
\node (a4) at (3,-.5){\tiny $\V{c}{3} = \{7\}$};
\node (inv) at (-4.1,5.25){\mbox{}};
\end{tikzpicture}\end{center}
Consequently, 
\begin{align*}
\CanonTop (\ttc(T)) &= (\{0\}, \{1,2,8\},  \{4,6\} , \{5\}, \{3\}, \{7\}). 
\end{align*}
\end{example}

Every recurrent configuration of a graph $G$ has a unique canonical toppling. The following lemma shows that each recurrent configuration corresponds to a unique minimal recurrent configuration, that is, the minimal recurrent configuration with the same canonical toppling.

\begin{lemma}\label{minlemma}
For every recurrent state $c'\in \Rec(G)$ there is a minimal recurrent state 
$\minrec(c') \in \MinRec(G)$, such that $\CanonTop(\minrec(c'))=\CanonTop(c')$.
We can compute $c=\minrec(c')$ by
\begin{equation}\label{eq:def_minc}
c_j = \begin{cases}
\big\vert \{ \ell \in \V{c'}{k} ~:~ j<\ell,\,k>i \} \big\vert &  \mbox{ if } j \in \U{c'}{i}, \\
\big\vert \{ \ell \in \U{c'}{k} ~:~  j>\ell,\,k\ge i \} \big\vert & \mbox{ if } j \in \V{c'}{i}. \\
\end{cases}
\end{equation}
Moreover, for any $c \in \MinRec(G)$, we have $c=\minrec(c)$.
\end{lemma}

\begin{proof}
  First we show that $c=\minrec(c')$ is a minimal recurrent state.  Suppose $j \in \V{c'}{i}$ and define $d_j^{(p)}$ to be the number of neighbors of the vertex~$j$ in block $\U{c'}{p}$ (all of which are smaller than $j$) for any $p$ and let $d_j$ be the degree of vertex $j$. The definition of $c_j$ in Equation~\eqref{eq:def_minc} can be rewritten as
\begin{equation}\label{eq:djval}c_j + \sum_{p=0}^{i-1} d_j^{(p)} = d_j.\end{equation}
This says that, starting from $c$, after toppling the vertices in $\U{c'}{0},\ldots,\U{c'}{i-1}$, the vertex $j$ would have exactly $d_j$ grains of sand (and thus be unstable).

A straightforward induction argument then shows that, starting from $c$, the vertices $[0,n]$ can be toppled in the order of $\CanonTop(c')$, returning to the configuration $c$ (that this returns to $c$ is a consequence of Dhar's criterion, Proposition~\ref{pro:DharBurning}).  Thus, by Proposition~\ref{pro:DharBurning}, the configuration $c$ is recurrent.  
To see that it is minimal note that $\sum_{j=1}^{n}c_j=\vert E(G)\vert-d_0$ because every
edge $(a,b)$ not incident to the sink is counted exactly once, namely by $c_a$ if $a$ topples before $b$ but otherwise by $c_b$.
This implies that $c$ is minimal with respect to its total number of grains (see e.g.~\cite{Biggs}).

Next we show that $\CanonTop(c)=\CanonTop(c')$. We know that $\U{c}{0}=\U{c'}{0}$ and proceed by induction.  Consider any $j\in\V{c'}{i}$ and assume $\U{c}{k}=\U{c'}{k}$ and $\V{c}{k}=\V{c'}{k}$ for all $k<i$.  We must show that $j\in\V{c}{i}$. By induction we know that $j$ is not in $\V{c}{k}$ for any $k<i$, and by Equation~\eqref{eq:djval} we know that vertex $j$ becomes unstable after toppling $\U{c'}{0},\ldots,\U{c'}{i-1}$ which is identical to $\U{c}{0},\ldots,\U{c}{i-1}$, and thus we must have $j\in\V{c}{i}$. An analogous argument holds for $j\in\U{c}{i}$. Therefore, $\CanonTop(c)=\CanonTop(c')$.

Finally, we show that if $c \in \MinRec(G)$, then $c=\minrec(c)$. Given any $j\in\V{c}{i}$ 
we know that 
\begin{equation}\label{pink1}
c_j + \sum_{p=0}^{i-1} d_j^{(p)} \ge d_j,
\end{equation}

as $j$ is toppling in block $\V{c}{i}$. We need to show that this is an equality, so suppose for a contradiction that this is not the case. From $c$ we can create a new configuration $c'$ by decreasing $c_j$ by one.  Then, starting from $c'$ and toppling the sink, the vertices of $G$ can be toppled in the order of $\CanonTop(c)$, and by Dhar's criterion~(see Proposition~\ref{pro:DharBurning}), the configuration $c'$ is recurrent. Since it has one fewer grains than $c$ in total, this contradicts the minimality of $c$. The case for $j \in \U{c}{i}$ is analogous.
\end{proof}

As an immediate consequence of Lemma~\ref{minlemma} we get the following result.

\begin{theorem}\label{thm:unique_minrec_canontop}
Given $c,c' \in \MinRec(G)$, we have $\CanonTop(c) = \CanonTop(c')$ if and only if $c=c'$.
\end{theorem}

\begin{example}\label{ex:minrec}
Let $F$ be the diagram in Example~\ref{exampleone} and $G=G(F)$.
Using the bijection from tableaux to configurations given in Definition~\ref{ttcdefn}, we have
$$\MinRec(G) = \{(0,0,1,0,2), \, (0,1,0,0,2), \, (0,1,1,0,1)\}.$$
On checking which of these minimal configurations may have more grains added to sites and remain recurrent, we find only one, namely $c=(0,1,1,0,2).$ The canonical toppling of this configuration is $\CanonTop(c) = (\{0\},\{1,3,5\}, \{2,4\})$.  On checking the canonical topplings of the three minimally recurrent configurations, we have
\begin{eqnarray*}
\CanonTop((0,0,1,0,2)) &=& (\{0\}, \{1,3,5\}, \{2,4\}) \\
\CanonTop((0,1,0,0,2)) &=& (\{0\}, \{1,5\}, \{2,4\}, \{3\}) \\
\CanonTop((0,1,1,0,1)) &=& (\{0\}, \{1,3\}, \{2\}, \{5\}, \{4\}).
\end{eqnarray*}
Consequently, $\minrec(c) = (0,0,1,0,2)$.
\end{example}

We may extend the notion of canonical toppling from minimal configurations to EW-tableaux by simply defining: $\CanonTop(T):=\CanonTop(\ttc(T))$. 
Moreover, we can give the following configuration-free definition of $\CanonTop(T)$ which is a simple consequence of emulating the topplings within the EW-tableau. 

\begin{definition}\label{canontopt}
Let $T \in \EW(F)$ and set $k=0$. Copy $T$ to $T'$ and do as follows:
\begin{itemize}
\item for all rows of only 1s in $T'$, record their labels in $\U{T}{k}$ and change all entries in those rows of $T'$ to 0;
\item for all columns of only 0s in $T'$, record their labels in $\V{T}{k+1}$ and change all entries in those columns of $T'$ to 1;
\end{itemize}
While the top row of $T'$ contains any 0s, increase $k$ by 1 and 
repeat both above steps.
With $k$ set to the last value it attained, define
$$
\CanonTop(T):= \left(\U{T}{0},\V{T}{1},\U{T}{1}, \ldots, \V{T}{k},\U{T}{k},\V{T}{k+1}\right).
$$
If the final set $\V{T}{k+1}$ is empty, omit it from the sequence.
\end{definition}

It follows from the construction in~\cite[Section~2]{sss} of the bijection from a  EW-tableaux $T$ to a permutation $\pi=\Psi(T)$ that $\CanonTop(T)$ in Definition~\ref{canontopt} equals the sequence of blocks in the \emph{run decomposition} of~$\pi$, which is the ordered partitioning of the letters of $\pi$ into maximal blocks that are alternately increasing and decreasing.  
We denote the run decomposition of a permutation $\pi$ by $\RunDec(\pi)$.
By, we prefix the run decomposition of~$\pi$ with a single block $\{0\}$ in order to match the initial block of $\CanonTop(T)$ (which is always equal to $\{0\}$), and  the first block after the initial 0-block is always increasing.

The only difference between the construction of $\CanonTop(T)$ and $\pi=\Psi(T)$ is that $\CanonTop(T)$ consists of a sequence of sets, whereas in the construction of $\pi$ we order the letters in those sets alternately increasingly and decreasingly, to specify a permutation. For example, if $\CanonTop(T)=(\{0\},\{4,8\},\{1,3,7\},\{6\}\{2,5\})$ then the corresponding permutation has run decomposition $\Psi(T)=0-48-731-6-52$. In particular, the canonical toppling of the corresponding Ferrers graph can be effected by first toppling vertex 0 and then toppling vertices in the order in which they appear in $\pi$, although it is irrelevant in which order vertices within a particular block are toppled.  We record this in the following theorem for future reference.

\begin{theorem}\label{thm:canontop_perm}
Regarding the blocks in the run decomposition of a permutation as unordered sets, we have, for any EW-tableau $T$:
\[\CanonTop(T)=\RunDec\Big(\Psi(T)\Big).\]
\end{theorem}

\subsection{Mapping minimal recurrent configurations to EW-tableaux and to permutations}\label{sec:map_ctt}

Using the notion of canonical toppling we can describe a map 
$$
\ctt:\MinRec(G(F))\mapsto \EW(F)
$$
which is the inverse of the map $\ttc$ described in Section~\ref{sec:map_ttc}.   The following proposition mirrors Corollary~14 in \cite{sss}, which says that the cell in row $r$ and column $c$ in a EW-tableau $T$ has a 1 if and only if $r$ precedes $c$ in $\pi=\Psi(T)$. It then follows from Proposition~13 in \cite{sss} that this map is the inverse of~$\ttc$.

\begin{prop}\label{prop:map_ctt} 
  Let $c \in \MinRec(G(F))$ and let $T$ be a tableau of shape $F$.  Label the rows and columns of $F$ in the usual way and let $\UU = \rowlabs(F)$ and $\VV=\collabs(F)$.  Letting $T=\ctt(c)$ be the tableau of shape $F$ having the following entries defines an inverse to the map $\ttc$:
$$T_{ij} = \begin{cases}
  0 & \mbox{if $i\in \UU$ topples after $j \in \VV$ in $\CanonTop(c)$,} \\
  1 & \mbox{if $i\in \UU$ topples before $j \in \VV$ in $\CanonTop(c)$}.
\end{cases} $$
\end{prop}

\begin{example}\label{topexamp}
  Let $c=(0,0,2,1,0,0,3,2)$ and let $F$ be the Ferrers diagram $(5,3,3,2)$.  The canonical toppling for this configuration is 
$$
\CanonTop(c) = (\{0\},\{1,2,7\},\{3\},\{8\},\{4,6\},\{5\}).
$$

Now $T_{0j}=1$ for all $j$ since $0$ topples before all other vertices. 
Similarly $T_{35} = 1$ since 3 topples before 5, and $T_{37}=0$ since 3 topples after 7, and $T_{38}=1$ since 3 topples before 8 does.
For the third row: $T_{45}=1$, $T_{47}=0$, and $T_{48}=0$.
For the fourth row: $T_{67}=0$ since 6 topples after 7 does, and $T_{68}=0$ for the same reason.
Therefore $\ctt(c)$ is:
\begin{center}
\EWdiagram{11111,101,001,00}{0,3,4,6}{8,7,5,2,1}
\end{center}
The permutation corresponding to this EW-tableau is 12738645, whose run decomposition is given by 0-127-3-8-64-5.
\end{example}

\begin{remark}\label{rem:x-patterns}
  Proposition~\ref{prop:map_ctt} allows us to see why the patterns $\begin{smallmatrix} 0&1 \\ 1&0 \end{smallmatrix}$ and $\begin{smallmatrix} 1&0 \\ 0&1 \end{smallmatrix}$ are forbidden for EW-tableaux, given their one-to-one correspondence to minimal recurrent configurations on Ferrers graphs.  
Namely, suppose that there is a configuration $c \in \MinRec(G(F))$ such that the tableau $T=\ctt(c)$ contains the pattern $\begin{smallmatrix} 0&1 \\ 1&0 \end{smallmatrix}$.  This implies that there exist values $i<\ell < k <j$ such that $T_{ij}=0$, $T_{ik}=1$, $T_{\ell j} = 1$, and $T_{\ell k} = 0$.  Definition~\ref{prop:map_ctt} then tells us that in the canonical toppling of $c$, $j$ topples before $i$, $i$ topples before $k$, $k$ topples before $\ell$ and $\ell$ topples before $j$, which leads to a contradiction. The case of the other pattern is analogous.
\end{remark}

By Proposition~\ref{prop:map_ctt}, for every Ferrers diagram $F$ we get a bijection 
$$
\zeta:\MinRec(G(F))\rightarrow\cls{n}(\rowlabs(F) \setminus \{0\}),
$$
where for $D\subset[1,n]$, $\cls{n}(D)$ is the set of permutations whose set of descent bottoms (letters preceded by a greater one) is $D$, and $\zeta=\Psi\cdot\ctt$, see Figure~\ref{fig:bijections}. Theorem~\ref{thm:canontop_perm} then implies the following corollary which gives an explicit description of $\zeta$.

\begin{figure}[ht]
\centering
\begin{tikzpicture}
\node (T) at (2,2){$\EW(F)$};
\node (P) at (0,0){$\mathsf{\cls{n}(\rowlabs(F) \setminus \{0\})}$};
\node (R) at (4,0){$\MinRec(G(F))$};
\draw[->] (1.7,1.7) -- (0.3,0.3);
\draw[->] (2.5,1.7) -- (4.2,0.3);
\draw[->] (3.8,0.3) -- (2.1,1.7);
\draw[->] (R) -- (P);
\node (T) at (.85,1.15){$\Psi$};
\node (P) at (3.7,1.25){$\ttc$};
\node (Q) at (2.4,1){$\ctt$};
\node (R) at (2,-.25){$\zeta$};
\end{tikzpicture}
\caption{A depiction of the bijections between EW-tableaux, minimal recurrent configurations and permutations, for a given Ferrers diagram $F$.\label{fig:bijections}}
\end{figure}
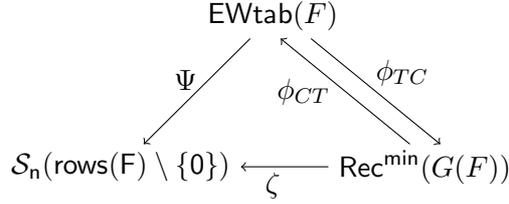

\begin{corollary}\label{cor:zeta}
For any $c\in\MinRec(G(F))$ we have
$$
\zeta(c)=\inc(\V{c}{1})\dec(\U{c}{1})\inc(\V{c}{2})\cdots, $$
where $\inc$ and $\dec$ denote  increasing and decreasing listings, respectively, of the sets in question.
\end{corollary}

\begin{example}\label{zetaex}
Suppose $c=(0,0,2,1,0,0,3,2)\in \MinRec(G(F))$ where $F$ is the Ferrers diagram with shape $(5,3,3,2)$ in Example~\ref{examplethree}. 
We have $\CanonTop(c) = (\{0\},\{1,2,7\},\{3\},\{8\},\{4,6\},\{5\})$, so $\zeta(c)=12738645$.
\end{example}


\section{Mapping recurrent configurations to decorated EW-tableaux}\label{sec:recstates_decEWT}
In this section we introduce decorated EW-tableaux and use them to give a characterization of all recurrent configurations of the sandpile model on Ferrers graphs in Theorem~\ref{classif}.  In order to do the latter we will need to introduce the notion of \emph{supplementary tableaux}.

\begin{definition}\label{decEWT}
A {\emph{decorated EW-tableau}} $D$ is a pair $(T,(x_1,\ldots,x_n))$ where
$T \in \EW(F)$ and 
\begin{itemize}
\item $0 \leq x_i < \# {\mbox{ 0s in row $i$ of $F$}}$ if $i \in \rowlabs(F)$,
\item $0 \leq x_i < \# {\mbox{ 1s in column $i$ of $F$}}$ if $i \in \collabs(F)$.
\end{itemize}
Let $\DEW(F)$ be the set of all decorated EW-tableaux on $F$.
\end{definition}

We depict a decorated EW-tableau $(T,(x_1,\ldots,x_n))$ by drawing the tableau $T$ and placing $x_i$ at the end or bottom of the row or column labeled $i$, see Example~\ref{exampletwo}.

This definition of decorated tableaux is such that the decorations of the rows and columns indicate how many more grains are on that vertex in relation to the minimal recurrent configuration to which the EW-tableau corresponds.

A general recurrent configuration can be considered to be represented as a decorated EW-tableau for several different EW-tableaux. We will indicate a way to associate a unique EW-tableau with a general recurrent configuration, using the function $\minrec$.  The purpose of doing this is to construct a bijection.

\begin{example}\label{exampletwo}
Let $F$ be the Ferrers diagram of shape $(3,2,1)$, and recall from Example~\ref{exampleone}:
$$\EW(F)=\left\{  
  A={\scriptstyle\young(111,00,0)}, B={\scriptsize{\young(111,01,0)}}, C={\scriptsize{\young(111,10,0)}} \right\}.$$ We can construct six decorated EW-tableaux on $F$:

\begin{equation}\label{tab:decEW}
  \begin{tabular}{cccccc}

\decoEWdiagram{111,00,0}{}{}{}{2/0,1/0}{3/0,2/0,1/0} &  \decoEWdiagram{111,01,0}{}{}{}{2/0,1/0}{3/0,2/0,1/0} & \decoEWdiagram{111,10,0}{}{}{}{2/0,1/0}{3/0,2/0,1/0} \\
\scriptsize(A,(0,0,0,0,0)) &\scriptsize(B,(0,0,0,0,0)) &\scriptsize(C,(0,0,0,0,0))\\[1 em]
\decoEWdiagram{111,00,0}{}{}{}{2/1,1/0}{3/0,2/0,1/0} &  \decoEWdiagram{111,01,0}{}{}{}{2/0,1/0}{3/0,2/1,1/0} & \decoEWdiagram{111,10,0}{}{}{}{2/0,1/0}{3/1,2/0,1/0}\\
 \scriptsize(A,(0,1,0,0,0)) &\scriptsize(B(0,0,1,0,0)) &\scriptsize(C,(0,0,0,0,1))
  \end{tabular}
\end{equation}
Recall that in Example~\ref{ex:minrec} we saw that
$$\MinRec(G(F)) = \{(0,0,1,0,2), \, (0,1,0,0,2), \, (0,1,1,0,1)\},$$
and the only non-minimal recurrent configuration is $c=(0,1,1,0,2)$.  We can obtain $c$ from all three minimal recurrent configurations via the three decorated EW-tableaux on the bottom row of \eqref{tab:decEW}.
\end{example}

The following definition extends the definition of $T_{ij}$ in Proposition~\ref{prop:map_ctt}, since the restriction of $S(T)$ to the Ferrers shape $F$ gives $T$.

\begin{definition}\label{def:supptab} 
Let $T \in \EW(F)$ and label the rows and columns of $F$ as usual.
Let $\UU:= \rowlabs(F)$ and $\VV:=\collabs(F)$.
Then define $S=\tts(T)$ to be the rectangular $|\UU|\times |\VV|$  tableau having entries
$$S_{ij} = \begin{cases}
  0 & \mbox{if $i\in \UU$ appears after $j \in \VV$ in $\CanonTop(T)$} \\
  1 & \mbox{if $i\in \UU$ appears before $j \in \VV$ in $\CanonTop(T)$.}
\end{cases} $$ We call $S$ the {\emph{supplementary tableau (to $T$)}} and will indicate 
this using the notation $S=S(T)$.
\end{definition}

\begin{example}\label{ex:supptab}
Let $T$ be the $EW$-tableau of Example~\ref{topexamp}:
\begin{center}
\EWdiagram{11111,101,001,00}{0,3,4,6}{8,7,5,2,1}
\end{center}
We have $$\CanonTop(T) = (\{0\},\{1,2,7\},\{3\},\{8\},\{4,6\},\{5\}).$$
The tableau $S$ will have $4$ rows (since $\UU=\rowlabs(F)=\{0,3,4,6\}$) and $5$ columns (since $\VV=\collabs(F)=\{1,2,5,7,8\}$).
Then $S_{0i}=1$ for all $i \in \VV$ since $0$ appears before all such $i$. 
For the second row (whose label is 3), we see that 3 appears before $\{5,8\}$ in $\CanonTop(T)$, and that 3 appears after $\{1,2,7\}$, so we have: $S_{35}=S_{38}=1$ and $S_{31}=S_{32}=S_{37}=0$.
For the third row (whose label is 4), we see that 4 appears before 5, and that 4 appears after each of $\{1,2,7,8\}$, so we have $S_{45}=1$ and $S_{41}=S_{42}=S_{47}=S_{48}=0$.
For the fourth row (whose label is 6), we see that 6 appears before $\{5\}$, and 6 appears after $\{1,2,7,8\}$, so we have $S_{65}=1$ and $S_{61}=S_{62}=S_{67}=S_{68}=0$.
Therefore
\begin{center}
\EWdiagramlab{11111,10100,00100,00100}{0,3,4,6}{8,7,5,2,1}{$\tts(T)=$}
\end{center}
\end{example}

We can now use $\ctt$ to define a map from recurrent configurations to decorated EW-tableaux. 

\begin{definition}\label{def:map_rec_to_decEWT}
Given $x \in \Rec(G(F))$ where $F$ is a Ferrers diagram having semi-perimeter $n+1$, let $\dew(x) = (\ctt(\minrec(x)),(a_1,\ldots,a_n))$
where $a_i = x_i - \minrec(x)_i$ for all $i\in [1,n]$. Let $\RDEW(F)=\dew(\Rec(G(F)))$. 
\end{definition}

\begin{example}\label{ex:map_rec_to_decEWT}
Let $F$ be the Ferrers diagram of shape $(3,2,1)$, and consider $c=(0,1,1,0,2)$, the unique non-minimal recurrent configuration from Example~\ref{exampletwo}.  Then
\begin{align*}\dew((0,1,1,0,2))
= \left(
	\ctt{(0,0,1,0,2)}
	,
	(0,1,0,0,0)\right)
=\hskip -20pt\raisebox{-2.25em}{\decoEWdiagram{111,00,0}{}{}{}{2/1,1/0}{3/0,2/0,1/0}}
.\end{align*}
\end{example}

We wish to describe a set of canonical decorated EW-tableaux that correspond bijectively 
to all recurrent configurations.  The main result of this section is Theorem~\ref{classif}, in which we give a description of $\RDEW(F)$ in terms of these canonical decorations.

\begin{definition}\label{def:rightangle}
Let $T$ be a EW-tableau of shape $F$. Let $S=S(T)$ be the supplementary tableau
corresponding to $T$.  Let us call a pair of cells $a$ and $b$ in a tableau {\emph{non-attacking}} if they are in different rows and different columns.  We also say 
that $b$ is non-attacking with respect to $a$ if $a$ and $b$ is a non-attacking pair, or simply that $b$ is non-attacking if it is clear from context what $a$ is.
We say that an entry $x \in \{0,1\}$ in $T_{jk}$ is a {\emph{\rightangle}}\ entry if and only if there exists a non-attacking $\mycomp{x}\not=x$ in $S_{j'k'}$ such that $S_{j'k}$ and $S_{jk'}$ both contain $x$.  We say a cell is a \rightangle\ if it contains a \rightangle\ entry and we say $S_{j'k'}$ \emph{induces} $T_{jk}$ to be a \rightangle.
\end{definition}

\begin{example}\label{ex:rightangle}
Let $T$ be the following tableau whose supplementary tableau is shown to its right. The $1$ that appears at position having labels $(0,6)$ is a \rightangle\ entry as is indicated by
the shaded entries on the right.
\begin{center}
\EWdiagramlab{1111,010,110,01}{0,2,3,5}{7,6,4,1}{$T=$}
\begin{tikzpicture}
  \def\minix{0.21}
  \def\miniy{0.21}
\def\xx{0.435}
\def\bx{0}\def\by{0}
\def\initx{0.25}\def\inity{0.37}
\def\jnitx{0.31}\def\jnity{0.47}
\def\stepx{0.43}\def\stepy{0.43}
\fill[black!30!white] (-1.22-\minix,-0.355-\miniy) rectangle (-1.22+\minix,-0.355+\miniy);
\fill[black!30!white] (-1.22-\minix,-1.65-\miniy) rectangle (-1.22+\minix,-1.65+\miniy);
\fill[black!30!white] (-0.355-\minix,-0.355-\miniy) rectangle (-0.355+\minix,-0.355+\miniy);
\fill[black!30!white] (-0.355-\minix,-1.65-\miniy) rectangle (-0.355+\minix,-1.65+\miniy);
\draw (\bx,\by) node [anchor = north east]  {\young(1111,0100,1100,0100)};
\foreach \x/\y/\lab in {0/0/0,2/1/2,2/2/3,3/3/5}  
	\draw (\bx-\initx-4.75*\stepx,\by-\inity-\y*\stepy) node [anchor=west]{\tiny{\lab}};
\foreach \x/\y/\lab in {0/0/1,1/0/4,2/2/6,3/3/7}
	\draw (\bx-\jnitx-\x*\stepx,\by-\jnity+1.7*\stepy) node [anchor=north]{\tiny{\lab}}; 
\draw (-2.25,-1) node [anchor = east]{$S=$};
\end{tikzpicture}
\end{center}
\end{example}

In order to prove the main result of this section, we need two lemmas.

\begin{lemma}\label{lem:cs}
Suppose $F$ is a Ferrers diagram of semiperimeter $n+1$.
Let $T \in \EW(F)$ and $S=S(T)$.
\begin{enumerate}
\item[(a)] If $j \in \U{T}{\ell}$ and $j<k$, then $k \in \V{T}{\ell}$ if and only if $S_{jk}=0$ and this $0$ is not a \rightangle\ entry.
\item[(b)] If $j \in \V{T}{\ell}$ and $k<j$, then $k \in \U{T}{\ell-1}$ if and only if $S_{kj}=1$ and this $1$ is not a \rightangle\ entry.
\end{enumerate}
\end{lemma}

\begin{proof}
Let $F,T$ and $S$ be as stated in the lemma. Let us write $a \topbef b$ to indicate that $a$ appears before $b$ in $\CanonTop(T)$.

First suppose that $j \in \U{T}{\ell}$ for some $\ell$ and $j<k$.  Then $k \in \V{T}{\ell-1}$ if and only if $k \topbef j$ and there does not exist $(j',k') \in \UU \times\VV$ such that $k\topbef j' \topbef k' \topbef j$.  Since $j<k$ the condition $k \topbef j$ is equivalent to $T_{jk}=0$, and as $j<k$ we get $T_{jk}=S_{jk}$. Furthermore, $k\topbef j' \topbef k' \topbef j$ where $k,k' \in \VV$ and $j,j' \in \UU$ is equivalent to $S_{j'k}=S_{jk'}=S_{jk}=0$ and $S_{j'k'}=1$, which is equivalent to $S_{jk}$ being a \rightangle\ entry.  Thus
\begin{align*}
k \in \V{T}{\ell} &\iff (T_{jk}=0) \mbox{~and~} \left(\begin{array}{c}\nexists (j',k') \in \UU\backslash\{j\} \times\VV\backslash\{k\} \mbox{ s.t. }\\ k\topbef j' \topbef k' \topbef j\end{array}\right)\\
&\iff (S_{jk}=0) \mbox{~and~} \left(\begin{array}{c}\nexists (j',k') \in \UU\backslash\{j\} \times\VV\backslash\{k\} \mbox{ s.t. }\\ S_{j'k}=S_{jk'}=S_{jk}=0 \ \&\  S_{j'k'}=1\end{array}\right)\\
&\iff (S_{jk}=0) \mbox{ and  this 0 is not a  \rightangle\ entry}.
\end{align*}

Now consider the case that $j \in \V{T}{\ell}$ for some $\ell$ and $k<j$.
By an argument analogous to that above we get:
\begin{align*}
k \in \U{T}{\ell-1} &\iff (S_{kj}=1) \mbox{~and~} \left(\begin{array}{c}\nexists (k',j') \in \UU\backslash\{j\} \times\VV\backslash\{k\} \mbox{ s.t. }\\ k\topbef j' \topbef k' \topbef j\end{array}\right)\\
&\iff (S_{kj}=1) \mbox{~and~} \left(\begin{array}{c}\nexists (k',j') \in \UU\backslash\{j\} \times\VV\backslash\{k\} \mbox{ s.t. }\\ S_{kj'}=S_{k'j}=S_{kj}=1 \mbox{ and } S_{k'j'}=0\end{array}\right)\\
&\iff (S_{kj}=1) \mbox{ and this 1 is not a \rightangle\ entry},
\end{align*}
\end{proof}

In words, Lemma~\ref{lem:cs} states that the entries of $T$ which are not \rightangle\ entries are pairs $(i,j) \in \UU \times \VV$ such that $i$ and $j$ appear in consecutive blocks in $\CanonTop(T)$ (in either order).
In what follows, we will have reason to construct a sequence of numbers $(\mu_j(c))_{j \in [1,n]}$ from minimal recurrent configurations.

\begin{definition}\label{defn:muj}
Suppose $F$ is a Ferrers diagram of semiperimeter $n+1$.
Given $c \in \MinRec(G(F))$
and $j \in [1,n]$, we define 
\begin{equation}\label{eq:def_muc}
\mu_j(c) := 
\begin{cases}
\big\vert \{ k \in \V{c}{\ell} ~:~ k>j \} \big\vert &  \mbox{ if } j \in \U{c}{\ell} \\
\big\vert \{ k \in \U{c}{\ell-1} ~:~  k<j \} \big\vert & \mbox{ if } j \in \V{c}{\ell}. \\
\end{cases}
\end{equation}
\end{definition}

Combining Lemma~\ref{lem:cs} and Definition~\ref{defn:muj} gives the following result.

\begin{lemma}\label{munu}
Suppose $F$ is a Ferrers diagram of semiperimeter $n+1$.
Given $T \in \EW(F)$,
define
$$
\nu_j(T) := 
\begin{cases}
  |\{\mbox{non-\rightangle\ 0s in row labeled $j$}\}| & \mbox{if }j \in \rowlabs(T) \\
  |\{\mbox{non-\rightangle\ 1s in column labeled $j$}\}| & \mbox{if }j \in \collabs(T).
\end{cases}
$$
Then $\nu_j(T)=\mu_j(\ttc(T))$ for all $j \in [1,n]$.
      \end{lemma}

We can now present the main theorem of this section, which allows us to characterize the elements of $\RDEW(F)$.

\begin{theorem} \label{classif}
Let $D=(T,(a_1,\ldots,a_n)) \in \DEW(F)$ be a decorated EW-tableau. 
Then $D \in \RDEW(F)$ if and only if 
$a_j < \nu_j(T)$ for all $j \in [1,n]$.
\end{theorem}
\begin{proof}
Let $D=(T,a)$ be a decorated EW-tableau where $a=(a_1,\ldots,a_n)$.
Let $\tilde{c}:=\ttc(T)$ be the unique minimal recurrent configuration corresponding
to $T$ and define $c$ to be the configuration $c:=\tilde{c} + a$.
By the definition of $\RDEW(F)$, the fact that $\ctt$ is the inverse of $\ttc$ (Proposition~\ref{prop:map_ctt}) and Theorem~\ref{thm:unique_minrec_canontop}, we get the following equivalence:
\begin{align*}
D \in \RDEW(F)&\iff\dew(c) = (T,a)
\iff \ctt(\minrec(c))=T\\
&\iff\minrec(c)=\ttc(T)=\tilde{c}\\
&\iff\CanonTop(c)=\CanonTop(\tilde{c})
\end{align*}

So to prove the result it suffices to show that $\CanonTop(c) = \CanonTop(\tilde{c})$ if and only if $a_j < \nu_j(T)$, for all $j \in [1,n]$. First, suppose that $$\CanonTop(c) = \CanonTop(\tilde{c})= (\U{}{0},\V{}{1},\U{}{1},\ldots,\V{}{k},\U{}{k}),$$ where $\U{}{k}$ is possibly empty. Fix $j \in [1,n]$ and suppose $j \in \U{}{i}$ for some $i \geq 1$. For $p \in [1,k]$, define $d_j^{(p)}$ to be the number of neighbors of $j$ in block $\V{}{p}$, that is, $d_j^{(p)}:= |\{\ell \in \V{}{p} ~:~ j < \ell \}|.$
By definition of the canonical toppling, in order for the vertex $j$ to be in the block $\U{}{i}$, the vertex $j$ must be stable after toppling all vertices of $\V{}{1},\V{}{2},\ldots,\V{}{i-1}$.
Toppling all the vertices in $\V{}{p}$ causes vertex $j$ to gain exactly $d_j^{(p)}$ extra grains. Thus, the above implies that we must have
\begin{equation}\label{pink2}
c_j +\sum_{p=1}^{i-1} d_j^{(p)} < d_j.
\end{equation}
By definition $c_j = \tilde{c}_j + a_j$, and by Lemma~\ref{minlemma} we have $\tilde{c}_j = \sum_{p=i+1}^{k} d_j^{(p)}$.
Using this in Equation~\eqref{pink2} we get
\begin{align*}
c_j +\sum_{p=1}^{i-1} d_j^{(p)}\ 
&=\ \tilde{c}_j + a_j + \sum_{p=1}^{i-1} d_j^{(p)}\   
=\ a_j + \sum_{p=i+1}^{k} d_j^{(p)} +  \sum_{p=1}^{i-1} d_j^{(p)} \ \\
&=\ a_j + d_j - d_j^{(i)}\ <\ d_j.
\end{align*}
This simplifies to $a_j < d_j^{(i)}$. 
Since $j \in \U{}{i}$ and $d_j^{(i)} := | \{\ell \in \V{}{i} ~:~ j<\ell\}|$ we have 
$d_j^{(i)} = \mu_j(\tilde{c})  = \mu_j(\ttc(T)) = \nu_j(T)$, from Lemma~\ref{munu}, and the desired result follows.
An analogous argument holds for the case $j \in \V{}{i+1}$ with $i \geq 0$, using the fact that $j$ must be stable after toppling all vertices of $\U{}{0},\U{}{1},\ldots,\U{}{i-1}$.

We now show the converse. Suppose that $a_j < \nu_j(T)=\mu_j(\ttc(T))$ for all $j \in [1,n]$. We wish to show that $\CanonTop(c) = \CanonTop(\tilde{c})$.
Since $c_j \geq \tilde{c}_j$ for all $j \in [1,n]$, we have that starting from $c$, the vertices of $[0,n]$ can be toppled in the order of $\CanonTop(\tilde{c})$, returning to the configuration $c$.
By definition of the canonical toppling, this implies that if $j \in \U{c}{i}$ (resp. $j \in \V{c}{i}$) and $j \in \U{\tilde{c}}{i'}$ (resp. $j \in \V{\tilde{c}}{i'}$) for some $i,i'$, we must have $i \leq i'$.

Suppose that $j \in \U{c}{i}$ and $j \in \U{\tilde{c}}{i'}$ for some $i,i'\geq 1$.  By the above, to show that $i=i'$ it is sufficient to show that vertex $j$ is still stable after toppling all vertices in $\V{c}{1},\ldots,\V{c}{i-1}$.  By assumption, we have $a_j < \nu_j(T)=\mu_j(\ttc(T))$. But, by the above, this is equivalent to $ c_j +\sum_{p=1}^{i-1} d_j^{(p)} < d_j$, which says exactly that $j$ is still stable after toppling all vertices in $\V{c}{1},\ldots,\V{c}{i-1}$. Thus $i=i'$.  By an analogous argument, if $j \in \V{c}{i}$ and $j \in \V{\tilde{c}}{i'}$ for some $i,i'$, we have $i=i'$. This implies that $\CanonTop(c) = \CanonTop(\tilde{c})$, and thus the theorem is proved.
\end{proof}

\begin{remark}\label{rmk:ineq}
The inequalities from Equations~\eqref{pink1} and \eqref{pink2}, which arise from the notion of canonical toppling introduced in Section~\ref{sec:canontop}, were first derived in \cite[Section 2]{Dh-Ma} in the case of general graphs. In that work, the authors made use of these inequalities and Dhar's burning algorithm (Proposition ~\ref{pro:DharBurning}) to establish a bijection between the recurrent configurations of the ASM on a graph $G$ and the spanning trees of $G$. However, this bijection is non-canonical, since it relies on an arbitrary ordering of the neighbors of each vertex of $G$. Indeed in this general setting, such a canonical bijection does not exist, whereas in this paper our use of these inequalities in the more specific setting of Ferrers graphs allows for a description of the recurrent configurations in terms of EW-tableaux whose decoration is canonical.
\end{remark}

\begin{example}\label{ex:big}
Let $T$ be the following EW-tableau and $F$ the underlying Ferrers shape:
\begin{center}
\EWdiagram{1111111111111,00111000000,10111101110,101110,1011,001,101}{0,3,4,10,13,15,16}{19,18,17,14,12,11,9,8,7,6,5,2,1}
\end{center}
The canonical toppling associated with this is 
\begin{align*}\CanonTop(T) = (\{0\},\{1,2,5,9,18\},&\{4,13,16\},\{6,7,8,11\}, \{10\},\\ &\{19\}, \{3,15\}, \{12,14,17\}).
\end{align*}
The supplementary tableau is $S(T)$:
\begin{center}
\begin{tikzpicture}
\def\step{0.43}
\draw (0,0) node [anchor = north west]  {\young(1111111111111,0011100000000,1011110111000,1011100000000,1011110111000,0011100000000,1011110111000)};
\foreach[count=\y] \lab in {0,3,4,10,13,15,16}
	\draw (0.25,-\y*\step+0.05) node [anchor=east]{\tiny{\lab}};
\foreach[count=\x] \lab in {19,18,17,14,12,11,9,8,7,6,5,2,1}
	\draw (\x*\step-0.05,0.2) node [anchor=north]{\tiny{\lab}};
\draw [line width=0.7mm] (0.15,-0.15)--(0.15,-0.15-7*\step)--(0.15+3*\step,-0.15-7*\step)--(0.15+3*\step,-0.15-5*\step)--(0.15+4*\step,-0.15-5*\step)--(0.15+4*\step,-0.15-4*\step)--(0.15+6*\step,-0.15-4*\step)--(0.15+6*\step,-0.15-3*\step)--(0.15+11*\step,-0.15-3*\step)--(0.15+11*\step,-0.15-1*\step)--(0.15+13*\step,-0.15-1*\step)--(0.15+13*\step,-0.15)--(0.15,-0.15);
      \end{tikzpicture}
\end{center}
For every entry of $T$, if it is a \rightangle\ entry then we replace it with an empty cell in the diagram below. 
That is, by Lemma~\ref{lem:cs}, the entries we keep are those given by pairs which appear in consecutive blocks of $\CanonTop(T)$.
\begin{center}
\EWdiagram{\ 1\ \ \ \ 1\ \ \ 111,0\ 111\ \ \ \ \ \ ,\ 0\ \ \ 101110,1\ \ \ \ 0,\ 0\ \ ,0\ 1,\ 0\ }{0,3,4,10,13,15,16}{19,18,17,14,12,11,9,8,7,6,5,2,1}
\end{center}
Apply Theorem~\ref{classif} to find that $D=(T,a) \in \RDEW(F)$ if and only if $a_j\in [0, \nu_j(T))$ for all $j \in [1,19]$.
We must determine $\nu_j(T)$ for all $j \in [1,19]$. 
From the previous tableau we may easily count the number of non-\rightangle\ 0s in every row and number of non-\rightangle\ 1s in each column:
$$(\nu_1(T),\ldots,\nu_{19}(T)) = (1,1,1,3,1,1,1,1,1,1,1,1,1,1,1,1,2,1,1).$$
\end{example}

We can compute the \rightangle\ entries without computing the canonical toppling. To do so we first introduce a result that allows us to compute the entries $S_{jk}$ of a supplementary tableau.

\begin{prop}\label{pro:st}
Let $T \in \EW(F)$ and $S$ be the supplementary tableau to $T$. 
Let $j \in \rowlabs(F)$ and $k \in \collabs(F)$ with $j>k$. 
Then $S_{jk}=0$ if and only if there exists $k' \in \collabs(F)$ with $k'>j$ such that $T_{jk'}=0$ and:
\begin{equation}\label{tomcondition}
\mbox{for all $j'<k$, if $T_{j'k'} = 0$ then $T_{j'k}=0$.}
\end{equation}
Otherwise $S_{jk}=1$.  
\end{prop}
Condition \eqref{tomcondition} in Proposition~\ref{pro:st} says that for any entry in column $k'$ containing a $0$, if there is an entry in the same row in column $k$, then it must also be $0$.

\begin{proof}
We begin by showing that a column $k'$ satisfies \eqref{tomcondition} if and only if $k'$ appears after $k$ or in the same block as $k$ in $\CanonTop(T)$.
To show this, suppose that $k'$ appears before $k$, that is, $k' \in \V{T}{\ell}$ and $k \in \V{T}{\ell+t}$ for some $\ell,t>0$. 
Let $j'=\min (\U{T}{\ell+t-1})$. We have $j'<k,k'$ and $k'\topbef j' \topbef k$ in that order. Thus $T_{j'k'}=0$ and $T_{j'k}=1$, which contradicts (\ref{tomcondition}). Conversely, if there exists $j'<k$ such that $T_{j'k'} = 0$ and $T_{j'k}=1$, then by
definition $k'\topbef j' \topbef k$, and in particular $k'$ appears before $k$ in $\CanonTop(T)$.

Now suppose that $S_{jk}=0$. This means that $k \topbef j$, that is, $k \in \V{T}{\ell}$ and $j \in \U{T}{\ell+t}$ for $\ell,t\ge0$. Let $k'=\max(\V{T}{\ell+t})$. 
We have $k'>j$, $T_{jk'}=0$ (since $k'$ appears before $j$ in $\CanonTop(T)$), and $k'$ appears after, or in the same block as, $k$. 
Thus, $k'$ satisfies \eqref{tomcondition}.

Conversely, suppose there exists $k'>j$ such that $T_{jk'}=0$ and $k'$ satisfies~\eqref{tomcondition}.
We have that $k' \topbef j$, and $k'$ appears after, or
in the same block as, $k$. Thus $k \topbef j$, and $S_{jk}=0$, as desired.
\end{proof}

\begin{example}\label{ex:stthm}
Let $T$ be the following EW-tableau:
\begin{center}
\EWdiagram{1111,0100,0110,01}{0,1,2,5}{7,6,4,3}
\end{center}
We can compute $S_{45}$ using Proposition~\ref{pro:st}, so $j=5$ and $k=4$. The only $k'$ with $T_{5k'}=0$ is $k'=7$. We then check that for every $i'$ with $T_{i'7}=0$ we have $T_{i'4}=0$. However, this is not the case for $i'=2$ where $T_{27}=0$ but $T_{24}=1$. Therefore, Proposition~\ref{pro:st} implies $S_{45}=1$.

We can also compute $S_{35}$, so $j=5$, $k=3$ and again the only possible $k'$ is $k'=7$. This time we can see that for every $i'$ with $T_{i'7}=0$ we have $T_{i'3}=0$. Therefore, Proposition~\ref{pro:st} implies $S_{45}=0$.
\end{example}

We now show that if $T_{jk}$ is a \rightangle\ in a EW-tableau $T$, then we can always find a rectangle where at least three of the cells are within $T$.

\begin{lemma}\label{lem:3entry}
If $T_{jk}$ is a \rightangle\ in a EW-tableau $T$, then there exists a cell $T_{j'k'}$ which induces $T_{jk}$ to be a \rightangle, and either
\begin{itemize}
\item $T_{jk}=T_{jk'}=0$ and $T_{j'k'}=1$ are cells in $T$, or
\item $T_{jk}=T_{j'k}=1$ and $T_{j'k'}=0$ are cells in $T$.
\end{itemize}
\begin{proof}
The cell $T_{jk}$ is a \rightangle\ if and only if there exists a row $j'$ and column $k'$ with $j\topbef k'\topbef j'\topbef k$ or $k\topbef j'\topbef k'\topbef j$. So we know that $j\in \U{T}{a}$ and $k\in \V{T}{b}$ where $a<b-1$ or $b<a$. First suppose $b<a$, let $j'=\min(\U{T}{a-1})$ and $k'=\max(\V{T}{a})$. Then we have $k\topbef j'\topbef k'\topbef j$, also $k'>j$, $k'>j'$ and $j<k$, so we have $T_{jk}=T_{jk'}=0$, $T_{j'k'}=1$, and all three are cells in $T$.

Similarly, for $a<b-1$, we let $j'=\min(\U{T}{b-1})$ and $k'=\max(\V{T}{b-1})$. Then we have $j\topbef k'\topbef j'\topbef k$, also $k'>j'$, $j'<k$ and $j<k$, so we have $T_{jk}=T_{j'k}=1$, $T_{j'k'}=0$, and all three are cells in $T$. 
\end{proof}
\end{lemma}

\begin{corollary}\label{cor:cellcs}
The cell $T_{jk}$ is a \rightangle\ in $T$ if and only if there exists a non-attacking cell $T_{j'k'}$ such that:
\begin{itemize}
\item When $T_{jk}=0$, we have $T_{j'k'}=1$, $T_{jk'}=0$ and $S_{j'k}=0$,
\item When $T_{jk}=1$, we have $T_{j'k'}=0$, $T_{j'k}=1$ and $S_{jk'}=1$.
\end{itemize}
\end{corollary}

Combining Proposition~\ref{pro:st} and Corollary~\ref{cor:cellcs} allows us to compute the \rightangle\ entries in a EW-tableau $T$ without computing the canonical toppling, 
and the corresponding supplementary tableau. So we can find all \rightangle\ entries in 
the following  way:
\begin{itemize}
\item For every $T_{jk}=0$ look at all non-attacking cells $T_{j'k'}=1$, with $T_{jk'}=0$. If $T_{j'k}=0$, or $T_{j'k}=\epsilon$ but $S_{j'k}=0$ by Proposition~\ref{pro:st}, then $T_{jk}$ is a \rightangle.
\item For every $T_{jk}=1$ look at all non-attacking cells $T_{j'k'}=0$, with $T_{j'k}=1$. If $T_{jk'}=1$, or $T_{jk'}=\epsilon$ but $S_{jk'}=1$ by Proposition~\ref{pro:st}, then $T_{jk}$ is a \rightangle.
\end{itemize}

\begin{example}\label{examplefour}
Let $T$ be the EW-tableau below on the left. The tableau on the right is obtained from $T$ by deleting every \rightangle\ entry of $T$.
\begin{center}
\EWdiagram{1111,0000,010,01}{0,1,3,5}{7,6,4,2}
\hskip 50pt
\EWdiagram{1\ 11,\ 0\ \ ,010,01}{0,1,3,5}{7,6,4,2}
\end{center}

We can see that $T_{06}$ is a \rightangle\ induced by $T_{37}$, and that $T_{17}$ and $T_{14}$ are \rightangle s induced by $T_{36}$.

The cell $T_{12}=1$ is also a \rightangle\ induced by $T_{36}$. To see this we need to apply Proposition~\ref{pro:st}
 to $S_{32}$. Let $k'=4$ so that we have $k'>3$ and $T_{3k'}=0$. Then the only $j'<2$ such that $T_{j'k'}=0$ is $j'=1$, and since $T_{12}=0$
 Proposition~\ref{pro:st} implies $S_{32}=0$. 

Therefore, we can see that $(\nu_1(T),\ldots,\nu_{7}(T))=(1,1,2,1,1,2,1)$.
\end{example}


\section{Decorated permutations}\label{sec:decperms}

\subsection{Representing recurrent configurations as decorated permutations}\label{sec:rec_perms}

We can represent a decorated EW-tableau by a \emph{decorated permutation} using the 
bijection between EW-tableaux and permutations introduced in~\cite{sss} (see 
Sections~\ref{sec:canontop} and \ref{sec:map_ctt} in the present paper). We will show in 
this section how to compute the minimum recurrent configuration corresponding to a EW-tableau as a simple statistic on the corresponding permutation.

\begin{definition}\label{def:decperm}
A \emph{(stable) decorated permutation} is a pair $(\pi,(a_1,\ldots,a_n))$, such that $\pi$ is an $n$-permutation and 
\begin{itemize}
\item $0\le a_i < |\{j\in \A{\pi}{k}\,:\,j>i,\,k\le\ell\}|$ if $i\in \D{\pi}{\ell}$,
\item $0\le a_i < |\{j\in \D{\pi}{k}\,:\,j<i,\,k<\ell\}|$ if $i\in \A{\pi}{\ell}$,
\end{itemize}
where $\D{\pi}{0},\A{\pi}{1},\D{\pi}{1},\A{\pi}{2},\ldots = \RunDec(\pi)$, with the convention $\D{\pi}{0}=\{0\}$ (see end of Section~\ref{sec:canontop}).
\end{definition}
Note that $a_i$ corresponds to the letter $i$ in $\pi$, not to the letter $\pi_i$ in position~$i$. We will be constructing unstable decorated permutations, where the~$a_i$ can be any non-negative integers, but for ease of notation we refer simply to decorated permutations to denote stable decorated permutations unless otherwise stated.  

A letter $\pi_i$ in a permutation $\pi=\pi_1\pi_2\ldots \pi_n$ is an \emph{ascent top} if $i=1$ or $\pi_{i-1}<\pi_i$, and a \emph{descent bottom} if $\pi_{i-1}>\pi_i$. 
Definition~\ref{def:decperm} thus says that if $i$ is an ascent top (resp. descent bottom) then $a_i$ is any non-negative integer strictly smaller than the number of descent bottoms (resp. ascent tops) smaller (resp. greater) than~$i$ that appear to the left of $i$ in $\pi$.  Note that in the run decomposition of $\pi$ the $A_\pi^{(k)}$ consist of ascent tops and the $D_\pi^{(k)}$ of descent bottoms for $k\ge1$.

\begin{lemma}\label{lem:PermToEW}
A pair $(F,(a_1,\ldots,a_n))$ is a decorated EW-tableau if and only if $(\Psi(F),(a_1,\dots, a_n))$ is a decorated permutation.
\begin{proof}
Let $\pi=\Psi(F)$. We know that the ascent tops and descent bottoms of $\pi$ are exactly the labels of columns and rows of $F$, respectively. Suppose $i$ is a row of $F$, and thus a descent bottom of $\pi$. We need to show that the number of $0$s in row $i$ of $F$ is equal to the number of larger ascent tops to the left of $i$ in $\pi$.
Let $Z$ be the set of column labels $j$ such that the cell $(i,j)$ contains a $0$.
Recall from Definition~\ref{canontopt} and the subsequent two paragraphs how
$\pi$ is constructed from $T$.  The letter $i$ cannot appear in $\pi$ until all $0$s in row $i$ have been recorded and changed to $1$s in the process of Definition~\ref{canontopt}.
Therefore, for every $j \in Z$, we know that $j$ must appear to the left of $i$ in $\pi$. Moreover, if $(i,j)$ is a cell of $F$ then $j>i$. 
This completes the case where $i$ is a row, the argument for when $i$ is a column is analogous.
\end{proof}
\end{lemma}

We saw in Example~\ref{exampletwo} that different decorated EW-tableaux can represent the same recurrent configuration. The same is true for decorated permutations. However, introducing the notion of canonical decorated permutations we get a bijection to recurrent configurations, via the set $\RDEW(F)$.

\begin{definition}\label{def:canondecperm}
A \emph{canonical decorated permutation} is a decorated permutation $(\pi,(a_1,\ldots,a_n))$ satisfying $a_i < \mu_i(\pi)$, where
$$
\mu_i(\pi)= 
\begin{cases}
~\big\vert \{ j \in \A{\pi}{\ell}: \, j > i \} \big\vert, \qquad \mbox{if } i \in \D{\pi}{\ell}, \\
~\big\vert \{ j \in \D{\pi}{\ell-1}: \, j < i \} \big\vert, \qquad \mbox{if } i \in \A{\pi}{\ell}. \\
\end{cases}
$$
\end{definition}
The difference between this and Definition~\ref{def:decperm} is that the bound on the 
decoration of a letter $i$ in $\pi$ in a canonical decoration of $\pi$ only depends on the letters in the block immediately preceding the block of  $i$.

\begin{example}
  Given a decorated permutation $(\pi,(a_1,\ldots,a_n))$ we can represent it as $0^{\emptyset}\pi_1^{a_{\pi_1}} \pi_2^{a_{\pi_2}}\ldots\pi_n^{a_{\pi_n}}$. The initial $0$ is to represent the initial block $\D{\pi}{0}=\{0\}$ of $\RunDec(\pi)$.  For example, we write the decorated permutation $(358714962,(2,1,0,0,0,0,0,0,1))$ as $0^{\emptyset}-3^05^08^0-7^01^2-4^09^1-6^02^1$ (where we have inserted dashes between blocks of the run decomposition), which is the maximum canonical decoration of 358714962.  The maximum stable decoration of this permutation is represented as $0^{\emptyset}-3^05^08^0-7^01^2-4^19^2-6^12^4$, since there are two ascent tops (8 and 9) greater than 6 and preceding the block of 6, and five in the case of~2, and also two descent bottoms (0 and 1) smaller than 4 and preceding the block of 4, and three in the case of 9.
\end{example}

By Theorem~\ref{thm:canontop_perm}, Lemma~\ref{munu}, Theorem~\ref{classif} and Lemma~\ref{lem:PermToEW} we get the following corollary.

\begin{corollary}\label{cor:car_RDEW}
For any EW-Tableau $F$, we have that $(F,(a_1,\ldots,a_n))$ is in $\RDEW(F)$ if and only if $(\Psi(F),(a_1,\dots, a_n))$ is a canonical decorated permutation.
\end{corollary}

The minimal recurrent configuration corresponding to a decorated permutation $(\pi,a)$ is given by $\ttc(\Psi^{-1}(\pi))$, which by a slight abuse of notation we denote $\minrec(\pi)$.  Lemma~\ref{minlemma} gives the following result for decorated permutations:

\begin{lemma}\label{lem:perm_to_minrec}
Given any permutation $\pi$, we can compute $r=\minrec(\pi)$ by:
$$
r_i=\begin{cases}
~|\{j\in \A{\pi}{k}\,:\,j>i, ~k>\ell\}|,&\mbox{ if } i\in \D{\pi}{\ell},\\
~|\{j\in \D{\pi}{k}\,:\,j<i, ~k\ge\ell\}|,&\mbox{ if } i\in \A{\pi}{\ell}.
\end{cases}
$$
\end{lemma}

\subsection{Tracking topplings through the decorated permutations}\label{sec:top_perms}

We can track the stabilization process of a configuration on a Ferrers graph through the corresponding (possibly unstable) decorated permutations.  During this process we move letters between blocks of the run decomposition, and when we move an ascent top into a different ascent block we place it in the unique position which maintains the increasing order of the block, similarly we maintain the decreasing order when we move a descent bottom into a descent block. Note that the blocks moved into may have been empty, and thus are created in the process.
 
Consider a (possibly unstable) decorated permutation $(\pi,(a_1,\dots,a_n))$. Let $r=\minrec(\pi)$ and let $G$ be the associated Ferrers graph. The number of grains on vertex $k$ in $G$ is given by $r_k+a_k$. If $a_k\ge \mu_k(\pi)$ and the letter $k$ is in $\A{\pi}{1}$ or $\D{\pi}{1}$, then $k$ is unstable. If $k$ is not in $\A{\pi}{1}$ or $\D{\pi}{1}$, then we say $k$ is \emph{unsettled}, which means the decorated permutation is not canonical, so we need to move $k$ from $\A{\pi}{i}$ to $\A{\pi}{i-1}$ or from $\D{\pi}{i}$ to $\D{\pi}{i-1}$.

We can track the stabilization of a decorated permutation $(\pi,(a_1,\ldots,a_n))$ with unstable and/or unsettled letters in the following way:
\begin{algorithm}\label{alg:topple}
  \begin{enumerate}
  \item If there are no unsettled letters in $\pi$ go to Step~\ref{step1}, otherwise let $y$ be the leftmost unsettled letter of $\pi$ and proceed to Step~\ref{step7}.\label{step5}
  \item Decrease $a_y$ by $\mu_y(\pi)$. If $y$ is in $\A{\pi}{k}$, resp. $\D{\pi}{k}$, move it to $\A{\pi}{k-1}$, resp. $\D{\pi}{k-1}$. Go to Step~\ref{step5}.\label{step7}
  \item Let $U$ be the set of unstable letters in $\pi$. If $U$ is empty then terminate and return $\pi$.\label{step1}
  \item Let $x$ be the leftmost letter of $U$.\label{step2}
  \item If $x\in \D{\pi}{1}$, then for every $j\in\A{\pi}{1}$ with $j>x$ increase $a_j(\pi)$ by $1$.\label{step3}
  \item Decrease $a_x$ by $\mu_x(\pi)$.
  \item \begin{itemize}
    \item If $x\in \D{\pi}{1}$, then move $x$ into $\D{\pi}{k}$ where $k$ is the smallest value such that $x$ is to the right of all larger ascent tops.
    \item If $x\in \A{\pi}{1}$ then move $x$ into $\A{\pi}{k}$ where $k$ is the smallest value such that $x$ is to the right of all smaller descent bottoms.
    \end{itemize}
  \item If $\D{\pi}{1}=\emptyset$, then merge $\A{\pi}{1}$ and $\A{\pi}{2}$ into a single ascent block.
  \item Delete $x$ from $U$. If $U$ is non-empty go to Step~\ref{step2} otherwise go to Step~\ref{step5}.
  \end{enumerate}
\end{algorithm}

\begin{prop}\label{prop:perm_topp_tracking}
Consider a decorated permutation $\pi$ that is not necessarily stable, which represents a configuration of the sandpile model on the Ferrers graph $G$. When $\pi$ is the input to Algorithm~\ref{alg:topple} the output is the canonical decorated permutation corresponding to the stabilization of $G$. 
\begin{proof}
First we deal with all unsettled vertices so we can determine which vertices are unstable. If a vertex $y$ is unsettled we need to move it to the left, as it topples earlier in the 
canonical toppling. So we move $y$ to the next available block to the left of the same 
type (to a block $\A{\pi}{k}$ if $y$ is an ascent top, $\D{\pi}{k}$ if $y$ is a descent bottom). When we move $y$ we do not change the number of grains on vertex $y$, so $r_y+a_y$ must remain unchanged. However $r_y$ is going to increase by $\mu_y(\pi)$ as the block that was immediately left of $y$ is now to the right. Therefore, by decreasing $a_y$ by $\mu_y(\pi)$ we do not change $r_y+a_y$. This explains Step~\ref{step7}.

Next we deal with the unstable vertices $U$, which we can topple in any order, due to the abelian property of the sandpile model.  When we topple $x\in U$ we move it to the specified position because in the tableau we change all cells in that row/column to $0/1$, respectively. Thus, according to the map from tableaux to permutations $x$ is in the first block in $\pi$ to the right of all smaller rows/larger columns if $x$ is a column/row, respectively.

Before we move $x$ we need to work out which vertices it has added a grain to; these are the larger ascent tops if $x$ is a descent bottom or the smaller descent bottoms if $x$ is an ascent top. To increase the number of grains on $j$ we can increase $r_j$ or $a_j$. If $j$ is 
to the right of $x$ in $\pi$ then, due to the way $r_j$ is defined, the new location of $x$ increases its value by one, but if $j$ is to the left of $x$ in $\pi$ the new location of $x$ does not affect the value of $r_j$ so we increase $a_j$.

After we move $x$ we know $r_x$ has a value of $0$, and the number of grains removed equals the degree of vertex $x$
which is $r_x+\mu_x(\pi)$ because $x$ was in $\A{\pi}{1}$ or $\D{\pi}{1}$, 
hence we need to decrease $a_x$ by the specified amount. So we have now toppled vertex $x$.

We have shown that toppling an unstable letter $x$ gives a decorated permutation corresponding to the graph obtained by toppling the vertex $x$. Therefore, once the algorithm terminates we have a decorated permutation which corresponds to the stabilization of $G$. Furthermore, we have shown that moving unsettled letters according to Step~\ref{step7} results in a decorated permutation corresponding to the same graph, and once all unsettled letters have been moved this must be canonical. Hence the output of the algorithm is the canonical decorated permutation corresponding to the stabilization of~$G$.
\end{proof}
\end{prop}

\begin{example}\label{ex:topp_tracking}

In Figure~\ref{fig:toppling} we show the process of stabilizing a permutation, and the corresponding EW-tableau and configuration on a Ferrers graph, according to Algorithm~\ref{alg:topple}.

On the graph the vertices are written as the number of grains on that vertex, and the label of a vertex is above/below in parentheses. In the tableaux the numbers $a_i$ are given on the south-east boundary. On the permutations the superscript of each letter $i$ is $a_i$, and beneath each letter we have put the pair $\mu_i(\pi),\minrec(\pi)_i$ so we 
can see when a letter is unstable or unsettled. Also, the blocks are separated by dashes, and we are omitting the zero-block at the beginning of the run decomposition.

\begin{figure}[ht]
\centering
\begin{tikzpicture}[scale=0.815]
\def\x{1}\def\y{1}\def\s{.5}
\node (p1) at (0*\x,25*\y){\sdu{6}{3}{1}{1}\hs\sd{7}{3}{1}{0}\hs $\!-\!$\hs\sd{2}{2}{2}{1}\hs $\!-\!$\hs\sd{3}{1}{1}{0}\hs\sd{5}{2}{1}{0}\hs $\!-\!$\hs\sd{4}{0}{1}{0}\hs\sd{1}{0}{2}{1}};
\node (p2) at (0*\x,20*\y){\sd{7}{3}{1}{0}\hs $\!-\!$\hs\sdu{2}{3}{1}{1}\hs $\!-\!$\hs\sd{3}{1}{1}{0}\hs\sd{5}{2}{1}{0}\hs $\!-\!$\hs\sd{4}{1}{1}{0}\hs\sd{1}{1}{2}{1}\hs $\!-\!$\hs\sd{6}{0}{2}{0}};
\node (p3) at (0*\x,15*\y){\hs\sd{3}{2}{1}{0}\hs\sd{5}{3}{1}{0}\hs\sdu{7}{3}{1}{1}\hs $\!-\!$\hs\sd{4}{1}{2}{0}\hs\sd{1}{1}{3}{1}\hs $\!-\!$\hs\sd{6}{1}{2}{0}\hs $\!-\!$\hs\sd{2}{0}{1}{0}};
\node (p4) at (0*\x,10*\y){\hs\sd{3}{2}{1}{0}\hs\sd{5}{3}{1}{0}\hs $\!-\!$\hs\sd{4}{2}{1}{0}\hs\sd{1}{2}{2}{1}\hs $\!-\!$\hs\sd{6}{1}{2}{0}\hs $\!-\!$\hs\sd{2}{1}{1}{0}\hs $\!-\!$\hs\sd{7}{0}{1}{0}};
\node (t1) at (5*\x,25*\y){    
  \begin{tikzpicture}[scale=\s, every node/.style={scale=2*\s}]
		\draw[step=1,thick] (0,0) grid (3,4);
		\draw[step=1,thick] (3,1) grid (4,4);
		\node (0) at (-.25,3.5){\tiny $0$};
		\node (1) at (-.25,2.5){\tiny $1$};
		\node (2) at (-.25,1.5){\tiny $2$};
		\node (3) at (3.5,4.25){\tiny $3$};
		\node (4) at (-.25,0.5){\tiny $4$};
		\node (5) at (2.5,4.25){\tiny $5$};
		\node (6) at (1.5,4.25){\tiny \underline{$6$}};
		\node (7) at (0.5,4.25){\tiny $7$};
		\node (07) at (0.5,3.5){$1$};
		\node (06) at (1.5,3.5){$1$};
		\node (05) at (2.5,3.5){$1$};
		\node (03) at (3.5,3.5){$1$};
		\node (17) at (0.5,2.5){$0$};
		\node (16) at (1.5,2.5){$0$};
		\node (15) at (2.5,2.5){$0$};
		\node (13) at (3.5,2.5){$0$};
		\node (27) at (0.5,1.5){$0$};
		\node (26) at (1.5,1.5){$0$};
		\node (25) at (2.5,1.5){$1$};
		\node (23) at (3.5,1.5){$1$};
		\node (47) at (0.5,0.5){$0$};
		\node (46) at (1.5,0.5){$0$};
		\node (45) at (2.5,0.5){$0$};
		\node (se0) at (4.25,3.5){\tiny $0$};
		\node (se1) at (4.25,2.45){\tiny $1$};
		\node (se2) at (4.25,1.5){\tiny $1$};
		\node (se3) at (3.5,.75){\tiny $0$};
		\node (se4) at (3.25,0.45){\tiny $0$};
		\node (se5) at (2.5,-.25){\tiny $0$};
		\node (se6) at (1.5,-.25){\tiny $1$};
		\node (se7) at (0.5,-.25){\tiny $0$};
		\node (a2) at (6.5,0){};
		\node (a1) at (-2,0){};
              \end{tikzpicture}
};
\node (t2) at (5*\x,20*\y){    
  \begin{tikzpicture}[scale=\s, every node/.style={scale=2*\s}]
		\draw[step=1,thick] (0,0) grid (3,4);
		\draw[step=1,thick] (3,1) grid (4,4);
		\node (0) at (-.25,3.5){\tiny $0$};
		\node (1) at (-.25,2.5){\tiny $1$};
		\node (2) at (-.25,1.45){\tiny \underline{$2$}};
		\node (3) at (3.5,4.25){\tiny $3$};
		\node (4) at (-.25,0.5){\tiny $4$};
		\node (5) at (2.5,4.25){\tiny $5$};
		\node (6) at (1.5,4.25){\tiny $6$};
		\node (7) at (0.5,4.25){\tiny $7$};
		\node (07) at (0.5,3.5){$1$};
		\node (06) at (1.5,3.5){$1$};
		\node (05) at (2.5,3.5){$1$};
		\node (03) at (3.5,3.5){$1$};
		\node (17) at (0.5,2.5){$0$};
		\node (16) at (1.5,2.5){$1$};
		\node (15) at (2.5,2.5){$0$};
		\node (13) at (3.5,2.5){$0$};
		\node (27) at (0.5,1.5){$0$};
		\node (26) at (1.5,1.5){$1$};
		\node (25) at (2.5,1.5){$1$};
		\node (23) at (3.5,1.5){$1$};
		\node (47) at (0.5,0.5){$0$};
		\node (46) at (1.5,0.5){$1$};
		\node (45) at (2.5,0.5){$0$};
		\node (se0) at (4.25,3.5){\tiny $0$};
		\node (se1) at (4.25,2.45){\tiny $1$};
		\node (se2) at (4.25,1.5){\tiny $1$};
		\node (se3) at (3.5,.75){\tiny $0$};
		\node (se4) at (3.25,0.45){\tiny $0$};
		\node (se5) at (2.5,-.25){\tiny $0$};
		\node (se6) at (1.5,-.25){\tiny $0$};
		\node (se7) at (0.5,-.25){\tiny $0$};
		\node (a2) at (6.5,0){};
		\node (a1) at (-2,0){};
              \end{tikzpicture}
};
\node (t3) at (5*\x,15*\y){    
  \begin{tikzpicture}[scale=\s, every node/.style={scale=2*\s}]
		\draw[step=1,thick] (0,0) grid (3,4);
		\draw[step=1,thick] (3,1) grid (4,4);
		\node (0) at (-.25,3.5){\tiny $0$};
		\node (1) at (-.25,2.5){\tiny $1$};
		\node (2) at (-.25,1.5){\tiny $2$};
		\node (3) at (3.5,4.25){\tiny $3$};
		\node (4) at (-.25,0.5){\tiny $4$};
		\node (5) at (2.5,4.25){\tiny $5$};
		\node (6) at (1.5,4.25){\tiny $6$};
		\node (7) at (0.5,4.25){\tiny \underline{$7$}};
		\node (07) at (0.5,3.5){$1$};
		\node (06) at (1.5,3.5){$1$};
		\node (05) at (2.5,3.5){$1$};
		\node (03) at (3.5,3.5){$1$};
		\node (17) at (0.5,2.5){$0$};
		\node (16) at (1.5,2.5){$1$};
		\node (15) at (2.5,2.5){$0$};
		\node (13) at (3.5,2.5){$0$};
		\node (27) at (0.5,1.5){$0$};
		\node (26) at (1.5,1.5){$0$};
		\node (25) at (2.5,1.5){$0$};
		\node (23) at (3.5,1.5){$0$};
		\node (47) at (0.5,0.5){$0$};
		\node (46) at (1.5,0.5){$1$};
		\node (45) at (2.5,0.5){$0$};
		\node (se0) at (4.25,3.5){\tiny $0$};
		\node (se1) at (4.25,2.45){\tiny $1$};
		\node (se2) at (4.25,1.5){\tiny $0$};
		\node (se3) at (3.5,.75){\tiny $0$};
		\node (se4) at (3.25,0.45){\tiny $0$};
		\node (se5) at (2.5,-.25){\tiny $0$};
		\node (se6) at (1.5,-.25){\tiny $0$};
		\node (se7) at (0.5,-.25){\tiny $1$};
		\node (a2) at (6.5,0){};
		\node (a1) at (-2,0){};
              \end{tikzpicture}
};
\node (t4) at (5*\x,10*\y){    
  \begin{tikzpicture}[scale=\s, every node/.style={scale=2*\s}]
		\draw[step=1,thick] (0,0) grid (3,4);
		\draw[step=1,thick] (3,1) grid (4,4);
		\node (0) at (-.25,3.5){\tiny $0$};
		\node (1) at (-.25,2.45){\tiny $1$};
		\node (2) at (-.25,1.5){\tiny $2$};
		\node (3) at (3.5,4.25){\tiny $3$};
		\node (4) at (-.25,0.45){\tiny $4$};
		\node (5) at (2.5,4.25){\tiny $5$};
		\node (6) at (1.5,4.25){\tiny $6$};
		\node (7) at (0.5,4.25){\tiny $7$};
		\node (07) at (0.5,3.5){$1$};
		\node (06) at (1.5,3.5){$1$};
		\node (05) at (2.5,3.5){$1$};
		\node (03) at (3.5,3.5){$1$};
		\node (17) at (0.5,2.5){$1$};
		\node (16) at (1.5,2.5){$1$};
		\node (15) at (2.5,2.5){$0$};
		\node (13) at (3.5,2.5){$0$};
		\node (27) at (0.5,1.5){$1$};
		\node (26) at (1.5,1.5){$0$};
		\node (25) at (2.5,1.5){$0$};
		\node (23) at (3.5,1.5){$0$};
		\node (47) at (0.5,0.5){$1$};
		\node (46) at (1.5,0.5){$1$};
		\node (45) at (2.5,0.5){$0$};
		\node (se0) at (4.25,3.5){\tiny $0$};
		\node (se1) at (4.25,2.45){\tiny $1$};
		\node (se2) at (4.25,1.5){\tiny $0$};
		\node (se3) at (3.5,.75){\tiny $0$};
		\node (se4) at (3.25,0.45){\tiny $0$};
		\node (se5) at (2.5,-.25){\tiny $0$};
		\node (se6) at (1.5,-.25){\tiny $0$};
		\node (se7) at (0.5,-.25){\tiny $0$};
		\node (a2) at (6.5,0){};
		\node (a1) at (-2,0){};
              \end{tikzpicture}
};

\node (g1) at (10*\x,25*\y){	
  \begin{tikzpicture}[scale=\s, every node/.style={scale=2*\s}]
	\node[inner sep=.5mm,label={\tiny (0)}] (0) at (0,2){ $*$};
	\node[inner sep=.5mm,label={\tiny (1)}] (1) at (2,2){$1$};
	\node[inner sep=.5mm,label={\tiny (2)}] (2) at (4,2){$3$};
	\node[inner sep=.5mm,label={\tiny (4)}] (4) at (6,2){$0$};
	\node[inner sep=.5mm,label=below:{\tiny (7)}] (7) at (0,0){$3$};
	\node[inner sep=.5mm,label=below:{\tiny \underline{(6)}}] (6) at (2,0){$4$};
	\node[inner sep=.5mm,label=below:{\tiny (5)}] (5) at (4,0){$2$};
	\node[inner sep=.5mm,label=below:{\tiny (3)}] (3) at (6,0){$1$};
	\draw (0) -- (7) -- (1) -- (6) -- (0) -- (5) -- (1) -- (3) -- (0);
	\draw (4) -- (5) -- (2) -- (7) -- (4) -- (6) -- (2) -- (3);
      \end{tikzpicture}
};
\node (g2) at (10*\x,20*\y){	
  \begin{tikzpicture}[scale=\s, every node/.style={scale=2*\s}]
	\node[inner sep=.5mm,label={\tiny (0)}] (0) at (0,2){ $*$};
	\node[inner sep=.5mm,label={\tiny (1)}] (1) at (2,2){$2$};
	\node[inner sep=.5mm,label={\tiny \underline{(2)}}] (2) at (4,2){$4$};
	\node[inner sep=.5mm,label={\tiny (4)}] (4) at (6,2){$1$};
	\node[inner sep=.5mm,label=below:{\tiny (7)}] (7) at (0,0){$3$};
	\node[inner sep=.5mm,label=below:{\tiny (6)}] (6) at (2,0){$0$};
	\node[inner sep=.5mm,label=below:{\tiny (5)}] (5) at (4,0){$2$};
	\node[inner sep=.5mm,label=below:{\tiny (3)}] (3) at (6,0){$1$};
	\draw (0) -- (7) -- (1) -- (6) -- (0) -- (5) -- (1) -- (3) -- (0);
	\draw (4) -- (5) -- (2) -- (7) -- (4) -- (6) -- (2) -- (3);
      \end{tikzpicture}
};
\node (g3) at (10*\x,15*\y){	
  \begin{tikzpicture}[scale=\s, every node/.style={scale=2*\s}]
	\node[inner sep=.5mm,label={\tiny (0)}] (0) at (0,2){ $*$};
	\node[inner sep=.5mm,label={\tiny (1)}] (1) at (2,2){$2$};
	\node[inner sep=.5mm,label={\tiny (2)}] (2) at (4,2){$0$};
	\node[inner sep=.5mm,label={\tiny (4)}] (4) at (6,2){$1$};
	\node[inner sep=.5mm,label=below:{\tiny \underline{(7)}}] (7) at (0,0){$4$};
	\node[inner sep=.5mm,label=below:{\tiny (6)}] (6) at (2,0){$1$};
	\node[inner sep=.5mm,label=below:{\tiny (5)}] (5) at (4,0){$3$};
	\node[inner sep=.5mm,label=below:{\tiny (3)}] (3) at (6,0){$2$};
	\draw (0) -- (7) -- (1) -- (6) -- (0) -- (5) -- (1) -- (3) -- (0);
	\draw (4) -- (5) -- (2) -- (7) -- (4) -- (6) -- (2) -- (3);
      \end{tikzpicture}
};
\node (g4) at (10*\x,10*\y){	
  \begin{tikzpicture}[scale=\s, every node/.style={scale=2*\s}]
	\node[inner sep=.5mm,label={\tiny (0)}] (0) at (0,2){ $*$};
	\node[inner sep=.5mm,label={\tiny (1)}] (1) at (2,2){$3$};
	\node[inner sep=.5mm,label={\tiny (2)}] (2) at (4,2){$1$};
	\node[inner sep=.5mm,label={\tiny (4)}] (4) at (6,2){$2$};
	\node[inner sep=.5mm,label=below:{\tiny (7)}] (7) at (0,0){$0$};
	\node[inner sep=.5mm,label=below:{\tiny (6)}] (6) at (2,0){$1$};
	\node[inner sep=.5mm,label=below:{\tiny (5)}] (5) at (4,0){$3$};
	\node[inner sep=.5mm,label=below:{\tiny (3)}] (3) at (6,0){$2$};
	\draw (0) -- (7) -- (1) -- (6) -- (0) -- (5) -- (1) -- (3) -- (0);
	\draw (4) -- (5) -- (2) -- (7) -- (4) -- (6) -- (2) -- (3);
      \end{tikzpicture}
};
\draw[->] (p1) -- (p2);\draw[->] (p2) -- (p3);\draw[->] (p3) -- (p4);
\draw[->] (t1) -- (t2);\draw[->] (t2) -- (t3);\draw[->] (t3) -- (t4);
\draw[->] (g1) -- (g2);\draw[->] (g2) -- (g3);\draw[->] (g3) -- (g4);
\end{tikzpicture}
\caption{The toppling of an unstable configuration of the sandpile model on a Ferrers graph, tracked through the corresponding permutations and EW-tableaux and configurations on the graph. Note that the permutations and EW-tableaux have decorations that are unstable except in the final stable configuration. The underlined vertices are those being toppled at each step. The decorations on the vertices are given by superscripts on the permutations and by the  decoration on the south-east edges of the tableaux. The pair of numbers  beneath the letters of the permutations are $\mu_i(\pi),\minrec(\pi)_i$.\label{fig:toppling}}
\end{figure}
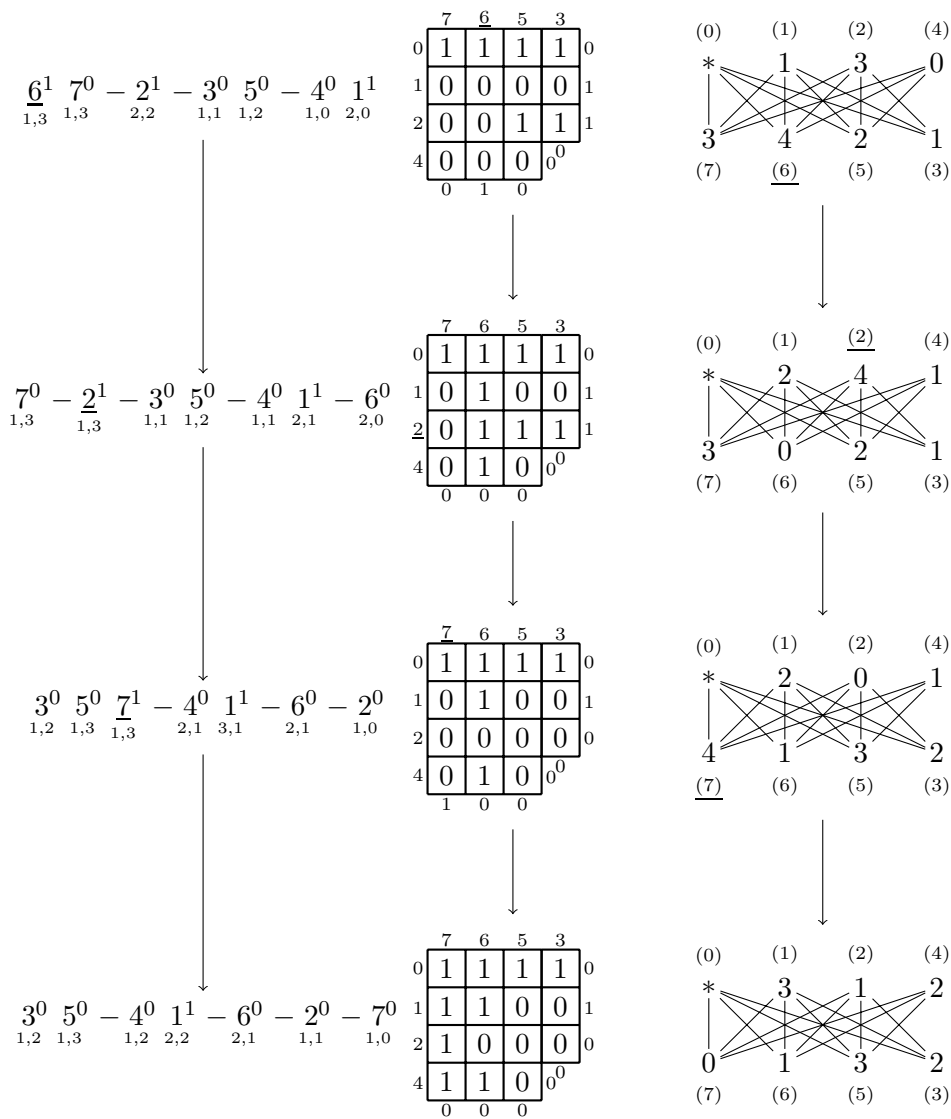

At first we have no unsettled vertices, so we go straight to dealing with the unstable vertices, the only such being $6$. As $6$ is an ascent top we can skip Step~\ref{step3}. We then decrease $a_6$ by $\mu_{6}(\pi)=1$, and then move $6$ to the first ascent block after all smaller descent bottoms, which is a new ascent block at the end of the permutation.

Next we need to deal with all unsettled vertices, but as there are none we go to the unstable vertices, of which the only one is $2$. We increase $a_i$ by one for all larger ascent tops to the left, which is only $7$. Then we decrease $a_2$ by $\mu_2(\pi)=1$ and move $2$ to the first block after all of $3,5$ and $6$. Finally, as $2$ was the last element in its block we must merge the first two ascent blocks, $\{7\}$ and $\{3,5\}$. We have now dealt with all unstable vertices and we repeat the same process again until we stabilize.
\end{example}

\subsection{A bijection between canonical decorated permutations and intransitive trees}
In \cite{postnikov-intransitive-trees}, Postnikov defines what he calls \emph{intransitive trees}, trees labeled with the numbers $1,2,\ldots,n$ such that the label of each vertex is either smaller than the labels of all its neighbors or greater than all of them.  The following describes a bijection from the set of canonical decorated $n$-permutations to the set of intransitive trees on $n+1$ vertices (see Figure~\ref{fig-perm-intrans}):

Let $\pi$ be an $n$-permutation.  The vertices of the corresponding tree $T$ are the letters $0,1,\ldots,n$, with $0$ designated as a root, and the depth of a vertex $v$ in $T$ being the number of edges on the unique path from $v$ to $0$, with the $k$-th \emph{level} of $T$ referring to the set of vertices of depth $k$ in $T$.

The letters of the first non-zero block in the run decomposition of $\pi$ are all children of the root.  More generally, the vertices at depth $k$ in $T$ are precisely the letters of the $k$-th block in $\pi$.  Let $\ell$ be a letter in the $k$-th block of $\pi$, with decoration $d$.  If $k$ is odd (resp. even) then $\ell$ is the child of the $d$-th smallest (resp. largest) letter at level $k-1$ in $T$, counting from $0$, so that the smallest (resp. largest) letter at level $k-1$ is the $0$-th smallest (resp. largest).

It is straightforward to check that the limitations on the decorations of letters of a canonical decorated permutation~$\pi$ guarantee that the above makes each vertex $v$ in $T$ the child of a smaller node if and only if the depth of $v$ is odd, so that $T$ is intransitive.

Conversely, it is straightforward to construct $\pi$ and its decoration from an intransitive tree~$T$.  Namely, the vertices of the $k$-th level of $T$ become the letters of the $k$-th block of $\pi$, ordered increasingly if $k$ is odd, decreasingly if $k$ is even.  The decoration of a letter $\ell$ is the number of siblings of its parent in $T$ that are smaller (resp. greater) than $\ell$ if the depth of $\ell$ in $T$ is odd (resp. even).

It is also straightforward to check that the previous paragraph describes the construction of a permutation whose blocks correspond to the levels of $T$ and whose decoration is legal. Finally, it is easy to see that this bijection maps the canonical toppling of the configuration corresponding to a decorated permutation to the breadth-first search of the corresponding intransitive tree, although the ordering of vertices at the same level is not reflected in the canonical toppling.

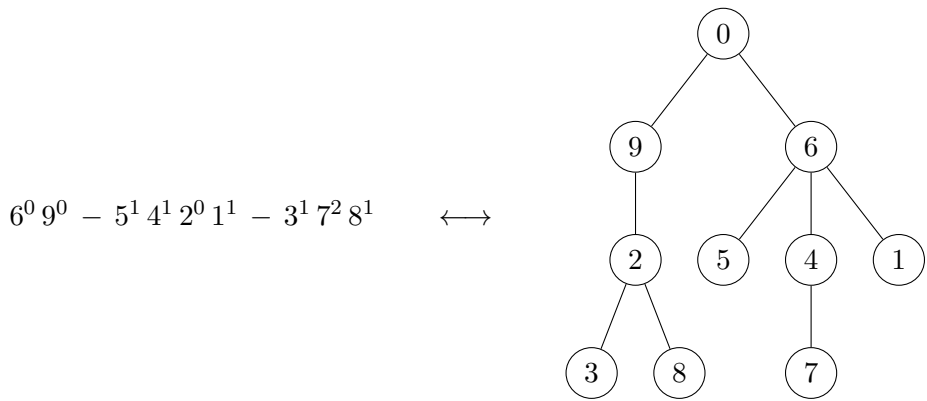
\begin{figure}[ht]
  \centering
\raisebox{6em}{
   $6^0\,9^0\,-\,5^1\,4^1\,2^0\,1^1\,-\,3^1\,7^2\,8^1$
   \hskip 20pt
   $\longleftrightarrow$
   \hskip 20pt
}
\begin{tikzpicture} [ level 1/.style={sibling distance=6em}, level 2/.style={sibling distance=3em}, every node/.style = {shape=circle, draw, align=center, top color=white}]
]
  \node {0}
    child { node {9} 
      child{node{2}
         child{node{3}}
         child{node{8}}}}
    child { node {6}       
       child { node {5} }
       child { node {4} 
          child { node{7}} }
       child { node {1} }
     };

\end{tikzpicture}
\caption{The intransitive tree corresponding to the permutation 695421378, decorated as shown.}\label{fig-perm-intrans}
\end{figure}

In \cite{postnikov-intransitive-trees}, Postnikov presents a bijection between intransitive trees and \emph{local binary search trees} (LBS trees).  The latter are labeled binary trees where each left child has a smaller label than its parent, and each right child a larger label than its parent. LBS trees were first considered by Gessel (private communication to Postnikov, see \cite{postnikov-intransitive-trees}).  It is straightforward to use this bijection to translate canonical decorated permutations, via our bijection described above, to local binary search trees, but this doesn't seem to offer anything more interesting than our bijection to intransitive trees.

\end{document}